\documentclass[12pt]{article}

\usepackage{amsmath, amsthm, amssymb, accents,amscd}
\usepackage{enumerate}
\usepackage{pdflscape}
\usepackage{caption}

\usepackage{ifpdf}
\ifpdf
\usepackage[pdftex]{graphicx}
\else
\usepackage[dvips]{graphicx}
\fi
\usepackage{tikz}
 	 \usetikzlibrary{arrows,backgrounds,math}
\usepackage[all]{xy}
\usepackage{tikz-cd}
\usepackage{multicol}

\input xy
\xyoption{all} 

\usepackage[pdftex,plainpages=false,hypertexnames=false,pdfpagelabels]{hyperref}
\newcommand{\arxiv}[1]{\href{http://arxiv.org/abs/#1}{\tt arXiv:\nolinkurl{#1}}}
\newcommand{\arXiv}[1]{\href{http://arxiv.org/abs/#1}{\tt arXiv:\nolinkurl{#1}}}

\newcommand{\googlebooks}[1]{(preview at \href{http://books.google.com/books?id=#1}{google books})}

\usepackage{xcolor}
\definecolor{dark-red}{rgb}{0.7,0.25,0.25}
\definecolor{dark-blue}{rgb}{0.15,0.15,0.55}
\definecolor{medium-blue}{rgb}{0,0,.8}
\definecolor{DarkGreen}{RGB}{0,150,0}
\definecolor{rho}{named}{red}
\hypersetup{
   colorlinks, linkcolor={purple},
   citecolor={medium-blue}, urlcolor={medium-blue}
}

\usepackage{longtable}
\usepackage{fullpage}

\theoremstyle{plain}
\newtheorem{thm}{Theorem}[section]
\newtheorem*{thm*}{Theorem}

\newtheorem{cor}[thm]{Corollary}

\newtheorem*{cor*}{Corollary}
\newtheorem{conj}[thm]{Conjecture}

\newtheorem*{conj*}{Conjecture}
\newtheorem*{lem*}{Lemma}
\newtheorem{lem}[thm]{Lemma}
\newtheorem{prop}[thm]{Proposition}

\newtheorem*{quest*}{Question}
\newtheorem*{claim*}{Claim}

\theoremstyle{definition}

\newtheorem{defn}[thm]{Definition}

\theoremstyle{remark}

\newtheorem{remark}[thm]{Remark}

\DeclareMathOperator{\tDiff}
{\mathrm{D}\!\widetilde{\,i\hspace{1.5mm}}\hspace{-1.5mm}\mathrm{ff}} 
\DeclareMathOperator{\Diff}{Diff}

\DeclareMathOperator{\Ann}{Ann}

\DeclareMathOperator{\tAnn}{A\widetilde{\!\!\phantom{i}n\phantom{i}\!\!}n}

\DeclareMathOperator{\Hom}{Hom}
\DeclareMathOperator{\Hol}{Hol}

\DeclareMathOperator{\id}{id}

\DeclareMathOperator{\Univ}{Univ}

\DeclareMathOperator{\Exp}{Exp}
\DeclareMathOperator{\Vir}{Vir}
\DeclareMathOperator{\Witt}{Witt}
\DeclareMathOperator{\supp}{supp}
\DeclareMathOperator{\Span}{Span}

\newcommand{\comment}[1]{}

\newcommand{\be}{\begin{enumerate}[label=(\arabic*)]}
\newcommand{\ee}{\end{enumerate}}

\newcommand{\DD}{\mathbb{D}}

\newcommand{\HH}{\mathbb{H}}

\newcommand{\CC}{\mathbb{C}}
\newcommand{\cC}{\mathcal{C}}

\renewcommand{\P}{\mathbb{P}}

\newcommand{\abs}[1]{\left| #1 \right|}

\def\semicolon{;}
\def\applytolist#1{
    \expandafter\def\csname multi#1\endcsname##1{
        \def\multiack{##1}\ifx\multiack\semicolon
            \def\next{\relax}
        \else
            \csname #1\endcsname{##1}
            \def\next{\csname multi#1\endcsname}
        \fi
        \next}
    \csname multi#1\endcsname}

\def\calc#1{\expandafter\def\csname c#1\endcsname{{\mathcal #1}}}
\applytolist{calc}QWERTYUIOPLKJHGFDSAZXCVBNM;
\def\bbc#1{\expandafter\def\csname bb#1\endcsname{{\mathbb #1}}}
\applytolist{bbc}QWERTYUIOPLKJHGFDSAZXCVBNM;
\def\bfc#1{\expandafter\def\csname bf#1\endcsname{{\mathbf #1}}}
\applytolist{bfc}QWERTYUIOPLKJHGFDSAZXCVBNM;
\def\sfc#1{\expandafter\def\csname s#1\endcsname{{\sf #1}}}
\applytolist{sfc}QWERTYUIOPLKJHGFDSAZXCVBNM;

\newcommand{\noshow}[1]{}
\newcommand{\MR}[1]{}

\usetikzlibrary{shapes}
\usetikzlibrary{backgrounds}
\usetikzlibrary{decorations,decorations.pathreplacing,decorations.markings}
\usetikzlibrary{fit,calc,through}
\usetikzlibrary{external}
\tikzset{
	super thick/.style={line width=3pt}
}
\tikzstyle{shaded}=[fill=red!10!blue!20!gray!30!white]
\tikzstyle{unshaded}=[fill=white]
\tikzstyle{empty box}=[circle, draw, thick, fill=white, opaque, inner sep=2mm]
\tikzstyle{annular}=[scale=.7, inner sep=1mm, baseline]
\tikzstyle{rectangular}=[scale=.75, inner sep=1mm, baseline=-.1cm]
\tikzstyle{mid>}=[decoration={markings, mark=at position 0.5 with {\arrow{>}}}, postaction={decorate}]
\tikzstyle{mid<}=[decoration={markings, mark=at position 0.5 with {\arrow{<}}}, postaction={decorate}]
\tikzstyle{over}=[double, draw=white, super thick, double=]

\title{The Segal-Neretin semigroup of annuli}
\author{Andr\'e G. Henriques, James E. Tener}
\date{}

\begin{document}

\maketitle

\abstract{
The Lie algebra of vector fields on $S^1$ integrates to the Lie group of diffeomorphisms of $S^1$.
But it is well known since the work of Segal and Neretin that there is no Lie group whose Lie algebra is the complexification of vector fields on $S^1$.
A substitute for that non-existent group is provided by the complex semigroup whose elements are \emph{annuli}: genus zero Riemann surfaces with two boundary circles parametrised by $S^1$.
The group $\mathrm{Diff}(S^1)$ sits at the boundary of that semigroup, and can be thought of as annuli which are completely ``thin'', i.e. with empty interior.

In this paper, we consider an enlargement of the semigroup of annuli, denoted $\Ann$, where the annuli are allowed to be ``partially thin'': their two boundary circles are allowed to touch each other along an arbitrary closed subset.
We prove that every (partially thin) annulus $A\in \Ann$ is the time-ordered exponential of a path with values in the cone of inward pointing complexified vector fields on $S^1$, and use that fact to construct a central extension
\[
0\to \bbC \times \bbZ \to \tAnn \to \Ann \to 0
\]
that integrates the universal (Virasoro) central extension of $\cX(S^1)$.

In later work, we will prove that every unitary positive energy representations of the Virasoro algebra integrates to a holomorphic representation of $\tAnn$ by bounded operators on a Hilbert space.
}

\setcounter{tocdepth}{1}
\tableofcontents

\section{Introduction}

The Virasoro Fr\'echet Lie algebra
\begin{align*}
&\Vir := \textstyle \bigg\{\sum_{n\in \bbZ} a_n L_n+ kC \,\bigg|\, 
\begin{aligned}
&a_n, k \in \bbC\\[-1mm]
&\text{$a_n$ is rapidly decreasing as $|n|\to\infty$}
\end{aligned}
\bigg\}
\intertext{with Lie bracket}
&[L_m,L_n] = (m-n)L_{n+m} + \frac{C}{12} (m^3-m) \delta_{n+m,0}.
\end{align*}
is famously known not to admit a Lie group that exponentiates it, even though
its real subalgebra $\Vir_\bbR := \Span(\{L_n - L_{-n},iL_n + iL_{-n}\}, iC)$ does exponentiate to a Lie group \cite{SegalDef,Neretin90}.
Specifically, $\Vir_\bbR$ exponentiates to a central extension $\tDiff(S^1)$ of $\Diff(S^1)$ known as the Bott-Virasoro group \cite{Bott77,FrenkelKim21ax,DebrayLiuWeis23}.
The Witt algebra
\[
\Witt := \Span(\{L_n\}_{n\in \mathbb Z}) = \Vir / \Span\{C\}
\]
is the complexification of the Lie algebra of $\Diff(S^1)$.
It exhibits the same feature of not admitting a Lie group that exponentiates it.

Remarkably, there does exist a semigroup that contains $\Diff(S^1)$ as its subset of invertible elements, and that behaves in many ways like the non-existent Lie group that exponentiates the Witt algebra.
Equivalently, it is the non-existent complexification of $\Diff(S^1)$.

That semigroup is known as the \emph{Segal-Neretin semigroup of annuli}.
It was independently discovered by Graeme Segal \cite{SegalDef} and Yuri Neretin \cite{Neretin90} in the eighties, while investigating two-dimensional conformal field theory (CFT).
Elements of the semigroup are isomorphism classes of complex annuli equipped with parametrisations of the boundary components by the unit circle $S^1 \subset \bbC$.
Composition is given by gluing (``conformal welding'') along the parametrisations.

In Segal's approach to CFT, the two-dimensional complex cobordism category $\mathrm{Cob}_2^{\mathit{conf}}$ is a central object of study, and CFTs are defined as functors out of that cobordism category.
The semigroup of annuli then naturally appears as the topologically trivial part of the complex cobordism category.
More precisely, the manifold $S^1$ is an object of the complex cobordism category.
The space $\Hom_{\mathrm{Cob}_2^{\mathit{conf}}}(S^1,S^1)$ splits as a disjoint union according to genus and number of connected components of the cobordism, and
the semigroup of annuli is the trivial connected component:
\[
\Ann^\circ \subset 
\Hom_{\mathrm{Cob}_2^{\mathit{conf}}}(S^1,S^1).
\]
In Segal's original formulation, the complex cobordism category $\mathrm{Cob}_2^{\mathit{conf}}$ did not have identity morphisms.
Identity morphisms can be formally added, and geometrically interpreted as annuli with thickness zero.

In this paper, we construct a version of the semigroup of annuli which contains annuli that are allowed to be `partially thin', meaning that the two boundary circles are allowed to intersect along an arbitrary closed subset.
Given a Jordan curve $\gamma$ (a parametrised closed simple curve $\gamma:S^1\to \bbC$ with winding number $1$), let us write $\mathrm{Int}(\gamma)\subset \bbC$ for the bounded component of the complement of the curve.

\begin{defn}
The semigroup of annuli $\Ann$ is the set of equivalence classes of pairs of smooth Jordan curves $\gamma_{in},\gamma_{out}:S^1\to \bbC$ satisfying $\mathrm{Int}(\gamma_{in})\subseteq \mathrm{Int}(\gamma_{out})$, where $(\gamma_{in},\gamma_{out})$ is deemed equivalent to $(\gamma'_{in},\gamma'_{out})$ if there exists a diffeomorphism $f:\bbC\to \bbC$ such that $f\circ \gamma_{in/out}=\gamma'_{in/out}$, and the restriction of $f$ to $\mathrm{Int}(\gamma_{out}) \setminus \overline{\mathrm{Int}(\gamma_{in})}$ is holomorphic.
\end{defn}\vspace{-.3cm}

The following is an example of an annulus:\hspace{1.3cm}
$
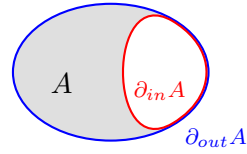
\tikz[baseline=0, yscale=.7]{
\draw[thick, blue, fill=gray!25] (0,0) circle (1.3);
\draw[thick, red, fill=white] (-50:1.273) arc (-50:50:1.273) to[in=-140, out=140, looseness=1.5] (-50:1.273);
\node at (-.65,-.2) {$A$};
\node[red, scale=1.05] at (.65,-.35) {$\scriptstyle \partial_{in}A$};
\node[blue, scale=1.05] at (1.4,-1.25) {$\scriptstyle \partial_{out}A$};
}
$\vspace{-.3cm}

\noindent
We also construct a central extension
\[
0\to \bbC \times \bbZ \to \tAnn \to \Ann \to 0
\]
of the semigroup of annuli which integrates the Virasoro central extension of the Witt algebra, and which we conjecture to be universal.
The following chart summarises the various Lie groups and semigroups mentioned above (on the left), along with their corresponding Lie algebras (on the right):

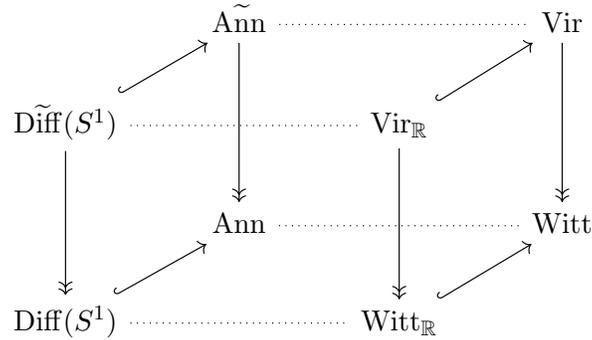
\[\begin{tikzcd}
&\tAnn \arrow[rr, -, dotted] \arrow[dd, ->>]  & & \Vir \arrow[dd, ->>] \\
\tDiff(S^1) \arrow[rr, -, dotted]\arrow[dd, ->>]\arrow[ur, hook] & & \Vir_\bbR \arrow[dd, ->>]\arrow[ur, hook]&\\
&\Ann \arrow[rr, -, dotted]  & & \Witt \\
\Diff(S^1) \arrow[rr, -, dotted]\arrow[ur, hook] & & \Witt_\bbR \arrow[ur, hook]&
\end{tikzcd}
\]
Both the semigroup of annuli $\Ann$ and its central extension $\tAnn$ are complex semigroups, meaning their tangent spaces are complex vector spaces, and the underlying spaces holomorphically embed into complex Fr\'echet vector spaces. But since neither $\Ann$ nor its central extension are manifolds (they are some kind of infinite dimensional manifolds with corners), formulating this requires care. In this paper, we rely on the technology of complex diffeological spaces, in which a `complex structure' is determined by the set of all holomorphic maps from finite-dimensional complex manifolds.

A drawback of this approach is that the complex diffeology `cannot see' the tangent spaces to $\Ann$ at non-interior points.
So we complement this with another diffeology, now based on smooth maps from finite-dimensional real manifolds with corners into $\Ann$. We then prove that both $\Ann$ and its central extension are diffeological semigroups for these two flavours of diffeology. Specifically, we prove that the semigroup law, provided by conformal welding, is compatible with the diffeologies.

The relationship between the semigroup of annuli and its Lie algebra requires some care,
as the exponential map is somewhat poorly behaved.
Even in the context of finite-dimensional Lie groups $G$, it is well-known that the exponential map $\mathfrak{g} \to G$ sometimes fails to be surjective. For infinite-dimensional Lie groups, the exponential map may even fail to be locally surjective\footnote{An example of a diffeomorphism not in the image of $\mathrm{exp}:\Witt \to \Diff(S^1)$ is a fixed-point free diffeomorphism which admits both finite orbits and infinite orbits. Those can be arranged to be as close to the identity as one wishes.}.
A fix for this lack of surjectivity is to consider the time-dependent exponential, which is a map
\begin{equation}
\label{eq: path in G}
\begin{split}
C^\infty(&[0,1],\mathfrak{g}) \longrightarrow G\\
&X  \mapsto   \prod_{1\ge \tau\ge 0} \Exp\big(X(\tau)d\tau\big)
\end{split}
\end{equation}
specified by
\begin{align*}
&\prod_{0 \ge \tau\ge 0} \Exp\big(X(\tau)d\tau\big) = 1 \\
&\frac{d}{dt} \bigg(\prod_{t \ge \tau\ge 0} \Exp\big(X(\tau)d\tau\big)\bigg) = X(t).\bigg(\prod_{t \ge \tau\ge 0} \Exp\big(X(\tau)d\tau\big)\bigg)
\end{align*}
where $X.g$ denotes the translation of $X \in \mathfrak{g}$ by the map which performs right multiplication by $g \in G$.

In $\Ann$, the tangent space at the identity is the cone $\Witt^{in} \subset \Witt$ of inwards pointing vector fields (possibly tangential to $S^1$):
\[
\Witt^{in} = \big\{ f(\theta) \tfrac{\partial}{\partial\theta} \in \Witt \,\big|\, \mathrm{Im}(f(\theta))\ge 0\big\}.
\]
Letting $\cP$ denote the set of smooth maps $[0,1]\to \Witt^{in}$, an annulus $A \in \Ann$ is the time-dependent exponential of a path $X \in \cP$ if there exists a smooth surjective map $h:S^1 \times [0,1] \to A$ such that $X = - \frac{\partial_t h(\theta,t)}{\partial_\theta h(\theta,t)}$.
We say that a path $X \in \cP$ is \emph{geometrically exponentiable} if such an annulus $A$ and map $h$ exist, in which case we write $A=\prod_{1\ge \tau\ge 0} \Exp\big(X(\tau)d\tau\big)$.
One of our main results is that the time-dependent exponential map is surjective:
\begin{thm}
For every annulus $A \in \Ann$ there exists a geometrically exponentiable path $X \in \cP$ such that
\[
A=\prod_{1\ge \tau\ge 0} \Exp\big(X(\tau)d\tau\big).
\]
\end{thm}
\noindent
We conjecture that all paths $[0,1]\to \Witt^{in}$ are geometrically exponentiable.

In future work, we will use this theorem to show that every unitary positive energy representation of the Virasoro algebra integrates to a holomorphic representation of the semigroup $\tAnn$.\footnote{Neretin establishes a similar result for the subsemigroup $\Ann^\circ$ of `thick annuli,' although the holomorphicity of the exponentiated representation is asserted without proof.}

We finish this introduction by listing a couple of things that we believe should be true, but that we were not able to establish:

\begin{conj}
The central extension 
\[
0\to \bbC \times \bbZ \to \tAnn \to \Ann \to 0
\]
is a universal central extension.
\end{conj}

\begin{conj}
\label{conj: Beltrami equation}
Every path $[0,1]\to \Witt^{in}$ is geometrically exponentiable
\end{conj}

Note that Conjecture~\ref{conj: Beltrami equation} is equivalent to a certain degenerate Beltrami equation always admitting solutions -- see Remark~\ref{rem: degenerate Beltrami}.

\subsection*{Acknowledgements}

The second author was supported by ARC Discovery Project DP200100067.

\section{Annuli with thin parts}
\label{sec: Annuli with thin parts}

Let us call a closed subset $D\subset \bbC$ bound by a smooth Jordan curve an \emph{embedded disc}.
An \emph{embedded annulus} is a subset of $\bbC$ of the form
\[
A=D_{out}\setminus \mathring D_{in}\subset \bbC,
\]
where $D_{in}\subset D_{out}\subset \bbC$ are embedded discs.
The inner and outer boundaries $\partial_{in}A=\partial D_{in}$ and $\partial_{out}A=\partial D_{out}$ of an embedded annulus are equipped with the orientations inherited from $D_{in/out}$.
Note that $\partial_{in}A\cap\partial_{out}A$ is allowed to be non-empty, and that the interior $\mathring A$ is allowed to be disconnected, or even empty.
We equip $A$ with the sheaf $\cO_A$ of functions that are continuous on $A$, holomorphic on $\mathring A$, and smooth on $\partial_{in}A$ and on $\partial_{out}A$. Namely, for an open $U\subset A$, we set
\begin{equation}\label{eq: def of O_A}
\cO_A(U) := \left\{
\parbox{9cm}
{continuous functions $U\to \bbC$ that are holomorphic on $\mathring A\cap U$, smooth on $\partial_{in}A\cap U$, and smooth on $\partial_{out}A\cap U$}
\right\}
\end{equation}

\begin{defn}\label{def: annulus}
An \emph{annulus} is a locally ringed space $A=(A,\cO_A)$ which is isomorphic to an embedded annulus.
We also record the orientations of $\partial_{in}A$ and $\partial_{out}A$ as part of the data of an annulus.\footnote{When the interior of $A$ is not empty, the orientations of $\partial_{in}A$ and $\partial_{out}A$ can be deduced from that of $\mathring A$.}
An annulus is called \emph{thick} if $\partial_{in}A\cap\partial_{out}A=\emptyset$ (in which case it is diffeomorphic to $S^1\times [0,1]$).
\end{defn}

An \emph{annulus with parametrised boundary} is an annulus equipped with orientation-preserving diffeomorphisms $\varphi_{out}:S^1\to \partial_{out}A$ and $\varphi_{in}:S^1\to \partial_{in}A$.

\begin{defn}
The \emph{semigroup of annuli} $\Ann$ is the set of isomorphisms classes of annuli with parametrised boundary.
\end{defn}

The set $\Ann^{\!\!\;\mathit{emb}}$ of embedded annuli with parametrised boundary is the subset of $C^\infty(S^1,\bbC)\times  C^\infty(S^1,\bbC)$
consisting of pairs of smooth Jordan curves such that one is `inside' the other.
By definition, any annulus is isomorphic to an embedded annulus, so there is an obvious surjective map $\Ann^{\!\!\;\mathit{emb}}\twoheadrightarrow\Ann$.
We thus get:
\begin{equation}
\label{eq: Ann as subquotient of C^inf x C^inf}
C^\infty(S^1,\bbC)\times  C^\infty(S^1,\bbC)
\,\supset\,
\Ann^{\!\!\;\mathit{emb}}
\,\twoheadrightarrow\,
\Ann.
\end{equation}

The semigroup operation on $\Ann$ is \emph{conformal welding} (explained in Section~\ref{sec: conformal welding}), which
is the gluing of annuli along their boundaries\footnote{This operation is just a pushout in the category of (locally) ringed spaces.}:
\begin{equation}\label{eq: 1st mention of conformal welding}
\begin{split}
\cup:\Ann &\times \Ann \,\longrightarrow\, \Ann\\
A_1 &\cup A_2\, :=\, \mathrm{pushout}\big(A_1 \xleftarrow{\,\varphi_{in}^{A_1}\,} S_1 \xrightarrow{\,\varphi_{out}^{A_2}\,} A_2\big)\\
&\phantom{\cup A_2\,\, :}=\, A_1 \sqcup A_2 \big/ \varphi_{in}^{A_1}(x)\sim \varphi_{out}^{A_2}(x).
\end{split}
\end{equation}
The annulus $A_1 \cup A_2$ is equipped with the sheaf $\cO_{A_{1} \cup A_{2}}$ specified by
$f\in \cO_{A_1 \cup A_2}(U) \Leftrightarrow f|_{U\cap A_i}\in \cO_{A_i}(U\cap A_i)$ for $i=1,2$.
We will prove later, in Proposition~\ref{prop: conformal welding of annuli}, that $(A_1 \cup A_2,\cO_{A_1\cup A_2})$ is again an element of $\Ann$.

The operation of conformal welding can also be performed between a disc (Definition~\ref{def: univ Teich space}) and an annulus, to form a new disc.
One can also weld two discs to the two boundaries of an annulus to for a closed genus zero surface.

Performing the latter operation and then applying the Riemann uniformization theorem to identify the resulting surface with $\bbC P^1$,
we can get an alternative description of the semigroup of annuli, as follows.
Let $\bbD_-:=\{z\in\bbC\cup\{\infty\}\,:\,|z|\ge1\}$ and $\bbD_+:=\bbD=\{z\in\bbC\,:\,|z|\le1\}$.
Given an annulus with parametrised boundary $A$, there is a unique biholomorphic map
\[
\psi\,:\,\bbD_-\cup A\cup \bbD_+\stackrel{\scriptscriptstyle\cong}\longrightarrow \bbC P^1
\]
sending $\infty\in\bbD_-$ to $\infty \in\bbC P^1$, with $\psi'(\infty)=1$, and $\psi''(\infty)=0$.
Letting $\psi_\pm:=\psi|_{\bbD_\pm}$, this identifies $\Ann$ with the space of pairs of embeddings 
\begin{equation} \label{eq: def: two discs mapping to CP^1}
\Ann \,\cong\, \left\{\,\,
\begin{matrix}
\psi_-:\bbD_-\hookrightarrow \bbC P^1\\[1mm]
\psi_+:\bbD_+\hookrightarrow \bbC P^1\\
\end{matrix}\,\left|\,\,
\begin{matrix}
\psi_+(z)=a_0+a_1z+a_2z^2+\ldots\,\,\,\,\,\\[.5mm]
\psi_-(z)=z+b_1z^{-1}+b_2z^{-2}+\ldots\\[.5mm]
\,\psi_-(\mathring\bbD_-)\cap \psi_+(\mathring\bbD_+)=\emptyset\quad
\end{matrix}\right.\right\}.
\end{equation}
The map which sends an annulus with parametrised boundary $A$ to the sequences $((a_i)_{i\ge 0},(b_i)_{i\ge 1})$ in the right hand side of \eqref{eq: def: two discs mapping to CP^1}
identifies $\Ann$ with a subset of the space $s\oplus s$ of pairs of rapidly decreasing sequences (those with the property that the maps $\psi_-$ and $\psi_+$ are embeddings, and $\psi_-(\mathring\DD_-)$ and $\psi_+(\mathring\DD_+)$ are disjoint).

Summarising the above discussion, we have three equivalent descriptions of the semigroup of annuli:
\begin{enumerate}
\item\label{enum: def 1 on Ann}
$\Ann$ is a quotient of $\Ann^{\!\!\;\mathit{emb}}$ which is itself a subset of $ C^\infty(S^1,\bbC)\times  C^\infty(S^1,\bbC)$. This is explained in \eqref{eq: Ann as subquotient of C^inf x C^inf}.
\item\label{enum: def 2 on Ann} 
$\Ann$ is a subset of the space $s\oplus s$ of pairs of rapidly decreasing sequences (representing maps $\bbD_-\to \bbC P^1$ and $\bbD_+\to \bbC P^1$). See \eqref{eq: def: two discs mapping to CP^1}.
\item\label{enum: def 3 on Ann}
The `official' definition via abstract annuli (Definition~\ref{def: annulus}).
\end{enumerate}
The semigroup of annuli is equipped with an involution
\begin{equation}
\label{eq: dag}
\dagger: \Ann \to \Ann
\end{equation}
which sends an annulus $A$ to the same annulus equipped with the opposite complex structure: $f\in \cO_A(U)\Leftrightarrow \bar f\in \cO_{A^\dagger}(U)$. The incoming and outgoing circles are exchanged, but their parametrisations remain the same.

There is an embedding
\begin{equation}
\label{eq: Diff ---> Ann}
\Diff(S^1)\hookrightarrow \Ann
\end{equation}
which sends $\psi\in \Diff(S^1)$ to the completely thin annulus  $A_\psi:=(A_\psi=S^1,\varphi_{in}=\psi,$ $\varphi_{out}=\id)$.
This embedding exhibits the semigroup of annuli as a complexification of sorts of $\Diff(S^1)$.
If we replace $\psi$ above by a univalent map $\bbD\to\bbD$, then we get another important sub-semigroup:

\begin{defn}
The \emph{semigroup of univalent maps} $\Univ$ is the set of holomorphic embeddings $f:\bbD\to\bbD$ (the derivative of $f$ is non-zero everywhere, including on $\partial\bbD$).
The inclusion $\Univ\hookrightarrow \Ann$ sends a univalent map $f\in \Univ$ to the annulus
\[
A_f := \bbD \setminus f(\mathring \bbD),
\]
with boundary parametrisations $\varphi_{in}=f$ and $\varphi_{out}=\id$.

We also define $\Univ^\dagger\subset \Ann$ to be the image of $\Univ\subset \Ann$ under the involution~\eqref{eq: dag}.
\end{defn}

Much of the above discussion involving annuli admits a straightforward (and easier) analog for discs with parametrised boundaries:

\begin{defn}\label{def: univ Teich space}
A \emph{disc} is a complex manifold with boundary which is holomorphically equivalent to the closed unit disc.
A \emph{disc with parametrised boundary} is a disc $D$ equipped with an orientation preserving diffeomorphism $S^1\to\partial D$.

The \emph{universal Teichm\"uller space} $\cT$ is the set of isomorphism classes of discs with parametrised boundary.
\end{defn}

$\cT$ admits three descriptions that are analogous to our three descriptions of $\Ann$:
(1) as a quotient of the subset $\cT^{\!\!\;\mathit{emb}}\subset C^\infty(S^1,\bbC)$ consisting of Jordan curves with positive winding,
(2) as a subspace of the space of rapidly decreasing sequences (those sequences $(a_i)$ such that $F(z) = z^{-1} + a_1z + a_2z^2 + \ldots$ defines a univalent map),
and (3) the `official' definition via abstract discs (Definition~\ref{def: univ Teich space}).

The semigroup of annuli acts on the universal Teichm\"uller space by conformal welding, and the map
\begin{equation}\label{eq: action of Ann on T}
\Ann\to\cT:\,A\mapsto A\cup\bbD
\end{equation}
identifies $\cT$ with the quotient $\Ann/\Univ$ (the quotient of $\Ann$ by the equivalence relation generated by $A\cup B\sim A$, for all $B\in \Univ$).

Both the semigroup of annuli $\Ann$ and the universal Teichm\"uller space $\cT$ have natural complex structures
(\cite{SegalDef, Neretin90},\cite{Nag88}).
But since $\Ann$ is not a manifold (it is some kind of infinite dimensional manifold with corners), formulating this requires care.
In this paper, it shall be sufficient for us to treat $\Ann$ and $\cT$ as a complex diffeological spaces:

\begin{defn}
A \emph{complex diffeological space} is a set $X$ equipped with, for every finite dimensional complex manifold $M$, a subset $\mathrm{Hol}(M,X)\subset \Hom_{\mathsf{Set}}(M,X)$ of `holomorphic maps', such that the assignment $M\mapsto \mathrm{Hol}(M,X)$ is a sheaf.

A holomorphic map $f:X \to Y$ between complex diffeological spaces is one that respects the diffeologies, i.e., 
one such that $f \circ -$ maps $\mathrm{Hol}(M,X)$ to $\mathrm{Hol}(M,Y)$.
\end{defn}

To equip the semigroup of annuli with a complex diffeology, we need to state which maps $M\to\Ann$ we declare to be holomorphic.

Recall that, by definition, $\Ann$ is a quotient of the space of embedded annuli, which is itself a subset of $ C^\infty(S^1,\bbC)\times  C^\infty(S^1,\bbC)$.
The obvious complex diffeology on $ C^\infty(S^1,\bbC)\times  C^\infty(S^1,\bbC)$ restricts to a diffeology on the space of embedded annuli, and
we equip $\Ann$ with the quotient diffeology.
(For the reader unfamiliar with $ C^\infty$-valued holomorphic functions, we recall that
for $M$ a finite-dimensional complex manifold, and $N$ a compact finite-dimensional smooth manifold,
a map $f: M \to C^\infty(N)$ is holomorphic if and only if the corresponding map
$\tilde f:M \times N \to \bbC$ is smooth, and fiberwise holomorphic in the direction of $M$.)

The alternative description of $\Ann$ as a subset of the vector space $s\oplus s$ of pairs of rapidly decreasing sequences, provides another way of equipping $\Ann$ with a complex diffeology, by declaring a map $M\to\Ann$ to be holomorphic if the corresponding map $M\to s\oplus s$ is.
We will show later, in Lemma~\ref{lem: three diffeologies on Ann agree}, that these two diffeologies agree.

There is also a third way of talking about families of annuli, which does not make explicit reference to embeddings.
Given a finite dimensional complex manifold $M$, and a holomorphic map $f:M\to C^\infty(S^1,\bbC)\times  C^\infty(S^1,\bbC)$ that lands in $\Ann^{\!\!\;\mathit{emb}}$, we can form the ``total space''
\begin{equation}
\label{eq: underline A}
\underline A := \big\{
(z,m)\in \bbC\times M\,\big|\,
\parbox{6cm}
{$z$ is in the annulus defined by $f(m)$}
\big\}
\end{equation}
of the family. It is a closed subset of $M\times \bbC$, equipped with the sheaf of functions that are continuous, smooth on the inner boundary, smooth on the outer boundary, and holomorphic in the interior.
And it comes with boundary parametrisations $\varphi_{in/out}:S^1\times M\hookrightarrow \underline A$.

\begin{defn}\label{def: family of abstract annuli}
Let $M$ be a finite dimensional complex manifold.

A \emph{holomorphic family of annuli} is a locally ringed space $\underline A$ with a map $\pi:\underline A\to M$ and boundary parametrisations $\varphi_{in/out}:S^1\times M \hookrightarrow \underline{A}$ such that $(\underline A,\pi,\varphi_{in/out})$ is locally isomorphic to a family of the form \eqref{eq: underline A}.
\end{defn}

The same works for discs:

\begin{defn}\label{def: family of abstract discs}
A \emph{holomorphic family of discs} is a map of locally ringed spaces $\underline D\to M$ with boundary parametrisation $S^1\times M \hookrightarrow \underline{D}$, which is locally equivalent to one coming from a holomorphic map $M\to \cT^{\!\!\;\mathit{emb}}\subset C^\infty(S^1,\bbC)$ (recall that $\cT^{\!\!\;\mathit{emb}}$ is the space of Jordan curves with positive winding, which is an open subset of $C^\infty(S^1,\bbC)$).
\end{defn}

Note that if $\underline{A}\to M$ is a holomorphic family of annuli, then the boundary parametrisation $\varphi_{in/out}:S^1\times M\to \underline{A}$ are `holomorphic' in the sense that for each point $\theta\in S^1$, the corresponding map $\varphi_{in/out}(\theta,-):M\to \underline{A}$ is holomorphic.
The same statement also holds for discs.

So far, we have encountered three potentially distinct notions of holomorphic fami\-lies of annuli:
\begin{enumerate}
\item\label{enum: diffeology 1 on Ann} using the diffeology coming from $\Ann$ as a quotient of $\Ann^{\!\!\;\mathit{emb}}$, which is itself a subset of $ C^\infty(S^1,\bbC)\times  C^\infty(S^1,\bbC)$
\item\label{enum: diffeology 2 on Ann} using the diffeology coming from $\Ann$ as a subset of the space $s\oplus s$ of pairs of rapidly decreasing sequences (representing maps $\bbD_-\to \bbC P^1$ and $\bbD_+\to \bbC P^1$).
\item\label{enum: diffeology 3 on Ann} via families of abstract annuli, as described in Definition~\ref{def: family of abstract annuli} 
\end{enumerate}
Our next goal is it show that these three notions of holomorphic families agree. The proof relies on the holomorphicity of conformal welding, a result whose proof we defer until Sections~\ref{sec: conformal welding discs} and~\ref{sec: conformal welding annuli}.

\begin{lem}\label{lem: three diffeologies on Ann agree}
The above three notions of holomorphic families of annuli all agree.
\end{lem}

\begin{proof}
A family that is holomorphic in the sense of \ref{enum: diffeology 2 on Ann} is visibly holomorphic in the sense of \ref{enum: diffeology 1 on Ann}, by restricting the embeddings of $\bbD_{\pm}$ to their boundary.

We next show that a family $M \to \Ann$ that is holomorphic in the sense of \ref{enum: diffeology 1 on Ann} is holomorphic in the sense of \ref{enum: diffeology 3 on Ann}.
Being holomorphic for the quotient diffeology, there exists an open cover $\{U_i\}$ of $M$ along with lifts
$U_i \to \Ann^{\!\!\;\mathit{emb}}$.
Each such family has an associated total space, which lives in $U_i\times \bbC$, and the total space of our initial family $M \to \Ann$ is obtained by gluing those together.
This yields a holomorphic family of abstract annuli as in \ref{enum: diffeology 3 on Ann}.

Finally, we show that a holomorphic family of abstract annuli as in \ref{enum: diffeology 3 on Ann} induces a holomorphic map in the sense of \ref{enum: diffeology 2 on Ann}.
Given such a family $\underline{A} \to M$ (along with its boundary parametrisations), we perform the same construction as \eqref{eq: def: two discs mapping to CP^1}, fiberwise in the family.
That is, we weld $\bbD_{\pm}$ to the incoming and outgoing boundaries of each fibre to produce a family of complex manifolds isomorphic to $\bbC P^1$, then uniformize the resulting family while controlling the first and second derivatives at the point $\infty$:
\[
(\bbD_- \times M) \cup \underline A \cup (\bbD_+ \times M) \,\,\cong\,\, \bbC P^1 \times M.
\]
By Propositions~\ref{prop: conformal welding of annuli with disc --- families} and~\ref{prop: welding discs in holomorphic families}, this produces a holomorphic map of $M$ into the appropriate subspace of $\Hol(\bbD_+, \bbC P^1) \times \Hol(\bbD_-, \bbC P^1)$.
\end{proof}

Similarly to the case of annuli,
the three descriptions of $\cT$ encountered above yield three notions of holomorphic families of discs.
And, as in Lemma~\ref{lem: three diffeologies on Ann agree},
these three notions of holomorphic families of discs all agree.

We record the following definition for later usage:

\begin{defn}
A \emph{holomorphic family of points} inside a holomorphic family of annuli $\pi:\underline A\to M$ is a map of locally ringed spaces $M\to \underline A$ which is a section of $\pi$.

(A holomorphic family of points inside a holomorphic family of discs $\pi:\underline D\to M$ is a map of locally ringed spaces $M\to \underline D$ which is a section of $\pi$.)
\end{defn}

A drawback of the complex diffeology on the semigroup of annuli is that it doesn't see that $\Ann$ is connected.
In order to illustrate this difficulty, consider the closed upper half-plane $\bbH \subset \bbC$ as a complex diffeological space (with the complex diffeology inherited from $\bbC$).
If $M$ is a connected complex manifold, then, by the open mapping theorem, any holomorphic map $f:M \to \bbH$ whose image intersects the boundary $\partial\bbH = \bbR$ must be constant.
So, as a complex diffeological space $\bbH$ decomposes as the disjoint union of its interior $\mathring{\bbH}$ and an uncountable discrete family of points on the real line.
The same phenomenon happens with $\Ann$: as a complex diffeological space, it is disconnected, with every element of $\Diff(S^1)$ being its own connected component (see Corollary~\ref{cor: hol family of thin is constant} below for a proof).
In particular, the complex diffeology `cannot see' the tangent spaces to $\Ann$ at non-interior points.

Because of the above shortcoming of the complex diffeology on $\Ann$,
we complement it with another diffeology, based on maps from finite-dimensional real manifolds with corners.
This `diffeology with corners'
works in much the same way as the complex diffeology.
Once again, there are three equivalent approaches to defining that diffeology.
The first approach comes from viewing $\Ann$ as a quotient of $\Ann^{\!\!\;\mathit{emb}}$. The second approach comes from viewing it as a subset $s\oplus s$.
Finally, we can make sense of smooth families of abstract annuli (parametrised by a manifold with corners). For that, we proceed as above except that, in the definition of the sheaf of functions on the total space, we replace the condition of being holomorphic in the interior by the condition of being fiberwise holomorphic in the interior.
The composition, or welding, of smooth families of abstract annuli is once again defined as a pushout of ringed spaces (see Proposition~\ref{prop: conformal welding of annuli --- families}). 
And the proof of equivalence of the above three notions of smooth families of annuli proceeds along the same steps as for Lemma~\ref{lem: three diffeologies on Ann agree}.

As is the case for $\Ann$, the universal Teichm\"uller space $\cT$ admits both a complex diffeology and a real diffeology (based on manifold with corners if one wishes), each of which admits three equivalent descriptions:
(1) the quotient diffeology on the subset of $ C^\infty(S^1,\bbC)$ consisting of Jordan curves with positive winding,
(2) the subspace diffeology of the space of rapidly decreasing sequences corresponding to Laurent series that define univalent maps,
and (3) the via families of abstract discs (defined via ringed spaces).

The `diffeology with corners' on $\Ann$ allows for a definition of the tangent bundle of the semigroup of annuli:

\begin{defn}\label{defn: tangent spaces with corners}
(Adaptation of \cite[Def. 3.1]{ChristensenWu16}, \cite[\S3.3]{Hector95}
to the case of manifolds with corners.)
Let $X$ be a diffeological space on the site of manifolds with corners, let $x\in X$ be a point, and let $P_xX$ be the set of all smooth maps $[0,\infty)\to X$ that send $0$ to~$x$.
Then the tangent space $T_xX$ is the quotient of the free vector space $\bbR[P_xX]$ by the relation which declares $a\gamma_1+b\gamma_2=\gamma_3$
if there exists a smooth map $\Gamma:[0,\infty)^2\to X$ such that
$\gamma_1(t)=\Gamma(t,0)$,
$\gamma_2(t)=\Gamma(0,t)$, and
$\gamma_3(t)=\Gamma(at,bt)$.
\end{defn}

We will show later, in Proposition~\ref{prop: tangent space of Ann}, that, as stated in \cite[Prop 2.4]{SegalDef},
\begin{equation}
T_A\Ann
\,=\,\frac{\cX(\partial_{out}A)\oplus \cX(\partial_{in}A)}{\cX_{\mathrm{hol}}(A)}
\,=\,\frac{\cX(S^1)\oplus \cX(S^1)}{\cX_{\mathrm{hol}}(A)},
\end{equation}
where $\cX(S^1)$ denotes the space of complexified vector fields on the circle, and $\cX_{\mathrm{hol}}(A)$ denotes the space of holomorphic vector fields on $A$ (Definition~\ref{def: tangent bundle of annulus}).

We now supply the promised proof that if a holomorphic family of annuli contains an element of $\Diff(S^1)\subset \Ann$, then the family is locally constant around that point. More generally, if $M$ is a connected complex manifold, 
and $\underline A\to M$ be a holomorphic family of annuli, with fiber over $m\in M$ denoted $A_m$, then
we have:

\begin{lem}\label{lem: hol family of thin is constant}
Let $\underline A\to M$ be as above. 
If $Z_m:=\partial_{in}A_m\cap\partial_{out}A_m$ denotes the thin part of $A_m$,
then the subsets $\varphi_{in}^{-1}(Z_m)$ and $\varphi_{out}^{-1}(Z_m)$ are independent of $m$,
and so is the map $\varphi_{out}^{-1}\circ \varphi_{in}:\varphi_{in}^{-1}(Z_m)\to\varphi_{out}^{-1}(Z_m)$.
\end{lem}

\begin{proof}
By gluing two discs onto $A_m$ and uniformizing, we may assume that the annuli $A_m$ are embedded.

Pick a point $m_0\in M$ and two points $\theta_{in},\theta_{out}\in S^1$ such that $\varphi^{m_0}_{in}(\theta_{in})=\varphi^{m_0}_{out}(\theta_{out})$.
We need to show that $\varphi^{m}_{in}(\theta_{in})=\varphi^{m}_{out}(\theta_{out})$ for all $m\in M$.
By performing an appropriate translation, we may assume without loss of generality that $\varphi^{m}_{in}(\theta_{in})=0$ for all $m\in M$.
We need to show that $\varphi^{m}_{out}(\theta_{out})$ is also zero for all $m\in M$.

Since $0\in \partial_{in} A_m$ for every $m$, there exists a neighbourhood $U\subset M$ of $m_0$ such that 
$\cup_{m\in U} A_m$ is not an open neighbourhood of $0$.
If $f:m\mapsto \varphi^{m}_{out}(\theta_{out})$ was not constant then, by the open mapping theorem, $f(U)$ would be open.
But $0\in f(U)$, and $f(U)\subset \cup_{m\in U} A_m$. So $f(U)$ cannot be open, and $f$ must be constant.
\end{proof}

\begin{cor}\label{cor: hol family of thin is constant}
Let $M$ be a connected complex manifold. If $\underline A\to M$ is a holomorphic family of annuli with one fiber in $\Diff(S^1)\subset \Ann$, then the family is constant.
\end{cor}

In the same spirit as the above lemma and corollary, we have the following result.
Let $\underline D \to M$ be a complex family of discs
indexed by a connected manifold, and
boundary parametrisation $\varphi:S^1\times M\to \underline D$.

\begin{prop}\label{prop holomorphic families of points}
Let $M$, $\underline D$, $\varphi$ be as above, and
let $a:M\to \underline D$ be a holomorphic family of points.
Then either $a(m)\in\mathring{D}_m$ $\forall m\in M$, or $a(m)\in \partial D_m$ $\forall m\in M$.
Moreover, in the latter case, $\varphi_m^{-1}(a(m))$ is constant.

The same result holds for holomorphic family of points inside holomorphic families of annuli.
\end{prop}
\begin{proof}
By passing to an open cover of $M$, we may assume without loss of generality that $\underline D$ comes from a holomorphic family $f:M \times S^1 \to \bbC$ of smooth Jordan curves.
Writing $f_m:=f(m,-)$
and letting $D_m$ be the closed disc bound by $f_m$, we then have
\[
\underline D\,=\,\{(m,z)\in M\times \bbC:z\in D_m\}.
\]
Replacing $f(m,\theta)$ by $f(m,\theta)-a(m)$, we may further assume without loss of generality that $a$ is identically zero.

If $0\in \mathring D_m$ $\forall m\in M$, then we are done.
So let us assume $\exists m_0\in M$ s.t. $0\in \partial D_{m_0}$.
Let $\theta:= f_{m_0}^{-1}(0) \in S^1$. 
We claim that we then have $f_m(\theta)=0$ for all $m\in M$.
Indeed, suppose by contradiction that $f_m(\theta)$ is not identically zero as a function of $m$.
Pick a neighbourhood $U$ of $m_0$ such that $\widehat D:=\bigcup_{m\in U} D_{m}$ does not contain a neighbourhood of $0$.
By the open mapping theorem, the image of $U$ under $m \mapsto f_m(\theta)$ is open.
We conclude that there exists $m\in U$ such that $f_m(\theta)\not \in \widehat D$. Contradiction.

The analogous result for annuli is proven in an entirely parallel fashion.
\end{proof}

\section{Conformal welding}
\label{sec: conformal welding}
Recall from \eqref{eq: 1st mention of conformal welding} that the composition of annuli, or \emph{conformal welding}, can be defined as a pushout in the category of ringed spaces. In this section, we will prove that the result of welding two annuli is again an annulus. We will also show that the the welding of an annulus with a disc is a disc, and that the welding of two discs is a Riemann sphere.
We establish families versions of the above results, both for families indexed by complex manifolds, and for families indexed by real manifolds with corners.
We begin with conformal welding of discs.

\subsection{Conformal welding of discs}
 \label{sec: conformal welding discs}

The following theorem is well-known (see e.g. \cite[\S III.1.4]{Lehto87}
for the more general case of welding via quasi-symmetric maps). We prove it here in the case of smooth welding maps:

\begin{thm} \label{thm: conformal welding of discs}
(Conformal welding)\, \rm
Let $D_1$ and $D_2$ be discs, and let ${\phi:\partial D_1\to \partial D_2}$ be an orientation reversing diffeomorphism.
Then 
the ringed space $(D_1\cup_\phi D_2,\cO_{D_1\cup_\phi D_2})$, with structure sheaf
\begin{equation}\label{eq: glued sheaf discs}
\cO_{D_1\cup_\phi D_2}(U):=\big\{f:U\to \CC\,\big|\,f|_{U\cap A_i}\in\cO_{D_i}(U\cap D_i),\,\text{\small for}\;i=1,2\big\}
\end{equation}
is isomorphic to $(\bbC P^1, \cO_{\bbC P^1})$.
Moreover, the image of $\partial D_1$ (equivalently $\partial D_2$) inside $D_1\cup_\phi D_2$ is a smooth curve.
\end{thm}

\begin{proof}
By the Riemann mapping theorem,
we may assume without loss of generality that $D_1=\DD_-=\{z\in\bbC \cup\{\infty\}\,:\,|z|\ge1\}$,
$D_2=\DD_+=\{z\in\bbC \,:\,|z|\le1\}$, and
\begin{align*}
&\phi:S^1=\partial\DD_-\stackrel{\scriptscriptstyle\cong}\longrightarrow\partial\DD_+=S^1.
\end{align*}
We will construct a homeomorphism $f:\bbD_-\cup_\phi\DD_+\to \bbC P^1$ that is holomorphic on the interiors $\mathring\DD_-$ and $\mathring\DD_+$,
and smooth on $\DD_-$ and $\DD_+$ all the way to the boundary.
In order to force $f$ to be unique,
we will also insist that $f$ maps $\infty\in\DD_-$ to $\infty\in\CC P^1$, with first derivative $f'(\infty)=1$ and second derivative $f''(\infty)=0$.

Letting $f_\pm:=f|_{\DD_\pm}$, we may rephrase the problem as that of finding a pair of functions $f_+,f_-\in\cC^\infty(S^1)$
of the form
\[
\qquad f_+(z)=a_0+a_1z+a_2z^2+\ldots\qquad f_-(z)=z+b_1z^{-1}+b_2z^{-2}+\ldots
\]
that make the following diagram commute:
\begin{equation}\label{eq: conformal welding}
\tikz[scale=.55, baseline=-40]{
\fill[fill=gray!25, rounded corners=7.5, 
thick, scale=4.8]
(.1,0) +(180:.45) -- +(200:.5) -- +(220:.45) -- +(240:.5) -- +(260:.45) -- +(280:.5) -- +(300:.5) -- +(320:.45) -- +(340:.5) -- +(360:.45)
 -- +(20:.5) -- +(40:.5) -- +(60:.45) -- +(80:.5) -- +(100:.45) -- +(120:.5) -- +(140:.45) -- +(160:.5) -- cycle;
\draw[red, thick, fill=white] (.48,0) circle (1.3);
\node[scale=.9] at (2,1.25) {$\DD_-$};
\draw[thick] (6,0) circle (1.3);
\draw[thick] (11,0) circle (1.3);
\draw[blue, thick, fill=gray!25] (16,0) circle (1.3);
\node[scale=.9] at (16.4,.6) {$\DD_+$};
\node at (3.6,0) {$\supseteq$};
\node at (13.5,0) {$\subseteq$};
\node[scale=1.1] at (8.5,0) {$\longrightarrow$};
\node at (8.5,.6) {$\phi$};
\fill[fill=gray!25, rounded corners=7.5, 
thick, scale=4.8]
(.75,-1.05) +(180:.45) -- +(200:.5) -- +(220:.45) -- +(240:.5) -- +(260:.45) -- +(280:.5) -- +(300:.5) -- +(320:.45) -- +(340:.5) -- +(360:.45)
 -- +(20:.5) -- +(40:.5) -- +(60:.45) -- +(80:.5) -- +(100:.45) -- +(120:.5) -- +(140:.45) -- +(160:.5) -- cycle;
\begin{scope}[cm={.3,0,0,.3,(1.7,-5.95)}]
\draw[red, thick, fill=white, rounded corners=8, xshift=8]
(3,0)--(0,1)--(3,5)--(7,8)--(8,4)--(12,5)--(11,1)--(8,0)--(6,3)--cycle;
\draw[thick, rounded corners=8, xshift=500]
(3,0)--(0,1)--(3,5)--(7,8)--(8,4)--(12,5)--(11,1)--(8,0)--(6,3)--cycle;
\draw[blue, fill=gray!25,  thick, rounded corners=8, xshift=950]
(3,0)--(0,1)--(3,5)--(7,8)--(8,4)--(12,5)--(11,1)--(8,0)--(6,3)--cycle;
\end{scope}
\node at (6.6,-5) {$\supseteq$};
\node at (11.3,-5) {$\subseteq$};
\node[rotate=-60,scale=1.1] at (7.1,-2.7) {$\longrightarrow$};
\node[rotate=-60,scale=1.1] at (.3+1.55,-2.7) {$\longrightarrow$};
\node[rotate=-120,scale=1.1] at (10-.1,-2.7) {$\longrightarrow$};
\node[rotate=-120,scale=1.1] at (14.8-.1,-2.7) {$\longrightarrow$};
\node[scale=.9] at (7.1-.5,-2.7) {$f_-$};
\node[scale=.9] at (10+.65-.1,-2.7) {$f_+$};
\node[scale=.6, draw] (A) at (17,-3.7) {holomorphic};
\node[scale=.6, draw] (B) at (-1+.4,-3.8) {holomorphic};
\draw (A.north west) -- +(-.6,.42);
\draw (B.north east) -- +(.75,.5);
}
\end{equation}
Let $\cH=\big\{\sum_{n\ge 0} a_nz^n\big\}\subset \cC^\infty(S^1)$ be the Hardy space,
let $\cH^\bot=\big\{\sum_{n< 0} b_nz^n\big\}$ be its orthogonal complement\bigskip, and let
$T:\cC^\infty(S^1)\to \cC^\infty(S^1)$ be the Hilbert transform, defined by
\[
Tf \,:= \tfrac1{2\pi i}\;\text{\sc p.v.}\!\int_{w\in S^1}\frac{f(w)}{w-z}dw\,=\tikz[baseline = -3]{\draw(0,.1) -- (.5,.5)(0,-.1) -- (.5,-.5);}\,\,\begin{matrix} \tfrac12 f & \text{\rm if}\,\,\,f\in\cH^{\phantom{\bot}} \\[5mm] -\tfrac12f  & \text{\rm if}\,\,\,f \in\cH^\bot\end{matrix}
\]
where the principal value is defined as $\text{\sc p.v.}\!\int_{w\in S^1} := \underset{\varepsilon\to 0}\lim\int_{w\in S^1\!\!\;,\!\;|w-z|\ge \varepsilon}$.

Letting $V_\phi:\cC^\infty(S^1)\to \cC^\infty(S^1)$ be the operator of precomposition by $\phi$,
we may rewrite the conditions in
\eqref{eq: conformal welding} as\smallskip
\begin{equation}\label{eq2: conformal welding}
V_\phi f_+=f_-\quad\,\,\, T f_+ = \tfrac12 f_+\quad\,\,\, T f_- = z-\tfrac12 f_-\,.
\smallskip
\end{equation}
Any solution of these equations must satisfy $V_\phi TV_\phi^{-1}f_-=V_\phi Tf_+=\tfrac12 V_\phi f_+ =\tfrac12f_-$.
In particular, $(T-V_\phi TV_\phi^{-1})f_-=z-f_-$.
Hence
\begin{equation}\label{eq: (1-(T-V_phi TV_phi^-1))f_-=z}
\Big[1-(T-V_\phi TV_\phi^{-1})\Big](f_-)=z.
\end{equation}
The singular integral can be rewritten as
\begin{align*}
V_\phi TV_\phi^{-1}f(z)
&=TV_\phi^{-1}f(\phi(z))\\
&=\tfrac1{2\pi i}\;\text{\sc p.v.}\!\int \frac{V_\phi^{-1}f(w)}{w-\phi(z)}dw\\
&=\tfrac1{2\pi i}\;\text{\sc p.v.}\!\int \frac{f(\phi^{-1}(w))}{w-\phi(z)}dw
=\tfrac1{2\pi i}\;\text{\sc p.v.}\!\int \frac{f(u)}{\phi(u)-\phi(z)}\phi'(u)du,
\end{align*}
so the integral kernel of $T-V_\phi TV_\phi^{-1}$ is given by 
\[
K(z,w)=\frac1{w-z}-\frac{\phi'(w)}{\phi(w)-\phi(z)}.
\]
The singularities of $\frac1{w-z}$ and $\frac{\phi'(w)}{\phi(w)-\phi(z)}$ exactly cancel out, so $K(z,w)$ is a smooth function, and
$T-V_\phi TV_\phi^{-1}$ a smoothing operator.
For $\phi$ close enough to the identity, the operator $1-(T-V_\phi TV_\phi^{-1})$ is therefore invertible, and we may solve Equation \eqref{eq: (1-(T-V_phi TV_phi^-1))f_-=z}
by simply writing
\[
f_-:=\Big[1-(T-V_\phi TV_\phi^{-1})\Big]^{-1}(z).
\]
It remains to check that $f_-$ as above together with $f_+:=V_\phi^{-1}f_-$ form a solution of the conformal welding problem \eqref{eq2: conformal welding}.
This is obvious when $\phi$ is analytic (because then \eqref{eq: conformal welding} admits a unique solution, by the Riemann uniformization theorem), and follows by continuity for arbitrary smooth $\phi$.
This finishes the proof of the theorem for $\phi$ close to the identity.

When $\phi$ is not necessarily close to the identity, write it as a composite $\phi=\phi''\circ\phi'$ of an analytic diffeomorphism $\phi'$ followed by a diffeomorphism $\phi''$ that is close to the identity.
The surface $\DD_-\cup_\phi\DD_+$ can be obtained from $\DD_-\cup_{\phi'}\DD_+\cong\CC P^1$ by cutting it along an analytically embedded circle $S\subset \bbC P^1$ (the image of $\partial\DD_-$) and regluing the two halves using $\phi''$.
Locally around $S$, this is equivalent to the problem of constructing $\DD_-\cup_{\phi''}\DD_+$, which we have just solved above.

It remains to identify the sheaf \eqref{eq: glued sheaf discs} with the sheaf of holomorphic functions on the welded surface.
This follows from the general fact (a standard application of Morera's theorem) that if $S\subset \CC$ be a smooth curve, $U\subset \CC$ is an open subset, and $f:U\to \CC$ is a continuous function whose restriction to $U\setminus S$ is holomorphic, then $f$ is holomorphic.
\end{proof}


Our next goal is to show that conformal welding of discs may be performed in families, both complex and real.
All the steps in the proof of Theorem~\ref{thm: conformal welding of discs} may be performed in both complex and real families, with one notable exception.
The Riemann mapping theorem, invoked at the very beginning of the proof, only admits a version for real families (Theorem~\ref{thm Riemann mapping theorem in families}), but no version for complex families:

\begin{thm}[{\cite[Thm 28.1 (and Thm 8.2)]{Bell92}}]   \label{thm Riemann mapping theorem in families}
Let $f:M\times S^1\to \bbC$ be a smooth family of smooth Jordan curves, parametrised by some manifold $M$.
Let $m\mapsto a_m:M \to \bbC$ be smooth family of points in the interiors of the discs $D_m$ bounded by those curves, and let $m\mapsto t_m:M \to S^1$ be a smooth family of tangent directions at those points.
For each $m\in M$, let
\[
F_m:\bbD^2 \to D_m
\]
be the solution of the Riemann mapping problem uniquely specified by $F_m(0)=a_m$, and $\partial_z F_m(0)\in \mathbb R_+ t_m$.
Then the maps $F_m$ assemble to a map $F: M\times \bbD^2 \to \bbC$, which is fiberwise holomorphic in the interior, and smooth all the way to the boundary, with everywhere non-zero derivative. \hfill $\square$
\end{thm}

The version of Theorem~\ref{thm: conformal welding of discs} for families indexed by real manioflds (with corners) requires no new ideas:

\begin{prop}
\label{prop: welding discs in smooth families}
Let $\underline D_1\to M$ and $\underline D_2\to M$ be two smooth families of discs, and let ${\phi:\partial D_1\to \partial D_2}$ be an orientation reversing fiberwise diffeomorphism.
Then
\[
(\underline D_1\cup_\phi \underline D_2,\cO_{\underline D_1\cup_\phi \underline D_2})
\]
is isomorphic to $\CC P^1\times M$ together with its sheaf of fiberwise holomorphic smooth functions.
Moreover, the image of $\partial \underline D_1$ inside $\underline D_1\cup_\phi \underline D_2$ is a smooth family of smooth curves.
\end{prop}

\begin{proof}
By the families version of the Riemann mapping theorem (Theorem~\ref{thm Riemann mapping theorem in families}), we may assume without loss of generality that $\underline D_1$ and $\underline D_2$ are trivial families $\bbD\times M \to M$.
In the proof of Theorem~\ref{thm: conformal welding of discs}, the smoothing operator $T-V_\phi T V_\phi^{-1}$ depends smoothly on $\phi$,
and therefore so does the inverse of $1-(T-V_\phi T V_\phi^{-1})$.
The solutions $f_+$ and $f_-$ of the equations \eqref{eq2: conformal welding} therefore also depend smoothly on $\phi$.
So the map $\underline D_1 \cup_\phi \underline D_2 \to \CC P^1 \times M$ given on fibers by the conformal welding map of Theorem~\ref{thm: conformal welding of discs} is smooth, as is the family of curves obtained by applying $f_-$ fiberwise to $\partial \underline D_1$.

It remains to identify the welded sheaf (the structure sheaf on the pushout of ringed spaces) with the sheaf of fiberwise holomorphic functions on $\CC P^1\times M$.
This follows from the fact that if $S\subset \bbC$ is a smooth curve, $U\subset \bbC$ is an open subset, and $f:U\to \bbC$ is a continuous function whose restriction to $U\setminus S$ is holomorphic, then $f$ is holomorphic.
\end{proof}

We now treat the case of holomorphic families:

\begin{prop}
\label{prop: welding discs in holomorphic families}
Let $M$ be a complex manifold,
let $f_+:\underline D_+\to M$ and $f_-:\underline D_-\to M$ be holomorphic families of discs with parametrised boundary
(meaning that the boundary parametrisations 
$\phi_\pm: S^1\times M\to D_\pm$ are smooth, and the functions $\phi_\pm(\theta,-):M\to D_\pm$ are holomorphic), and let us
abbreviate $\underline D_+\cup_{\phi_-,S^1\times M,\phi_+} \underline D_-$ by $\underline D_+\cup_{\phi_\pm} \underline D_-$. Then the we have an isomorphism of ringed spaces
\[
\underline D_+\cup_{\phi_\pm} \underline D_-
\,\,\cong\,\, \CC P^1\times M,
\]
compatible with the projections to $M$.
Moreover, the map
\begin{equation}\label{eq: moreover the map...}
S^1\times M\xrightarrow{\phi_+} \partial \underline D_+
\hookrightarrow \underline D_+\cup_{\phi_\pm} \underline D_-\cong \CC P^1\times M
\end{equation}
is a holomorphic family of smooth curves.
\end{prop}

\begin{proof}
Let $D^m_\pm:=f_\pm^{-1}(m)$.
We assume without loss of generality that $\underline D_+$ and $\underline D_-$ are determined by holomorphic families of smooth Jordan curves $\varphi_\pm:m\mapsto \varphi_\pm^m:M\to C^\infty(S^1,\bbC)$,
in the sense that $D_+^m \subset \bbC$ is the disc bound by $\varphi_+^m$,
and $D_-^m \subset \bbC\cup \{\infty\}$ is the co-disc bound by $\varphi_-^m$.
We also assume that $0\in D_+^m$ for each $m\in M$.

For each $m\in M$,
the uniformising map $F^m : D_+^m \cup D_-^m \to \bbC P^1$ is the union of maps
$F_+^m : D_+^m \to \bbC$ and $F_-^m : D_-^m \to \bbC \cup \{\infty\}$ characterised by
\begin{align*}
&	F_+^m(z) = a_0 + a_1z + a_2z^2 + \ldots\,\,  \text{around zero, and analytically extends to $D_+^m$}
\\&	F_-^m(z) = z^{-1} + b_1z + b_2z^2 + \ldots\,\,  \text{around $\infty$,  and analytically extends to $D_-^m$}
\\&	F_+^m \circ \varphi_+^m = F_-^m \circ \varphi_-^m.
\end{align*}
Our goal is to show that the isomorphism
$F_+\cup F_-:D_+\cup_{\phi_\pm} D_-\cong \CC P^1\times M$
is holomorphic on each of the two pieces.
We do this by showing that the map $M \to \CC[[z]]$ which sends $m \in M$ to the power series expansion of $F_+^m$ about $0$ is holomorphic, and similarly for $F_-^m$.
Note that these maps are smooth by Proposition~\ref{prop: welding discs in smooth families}, so it suffices to verify that the induced map on tangent spaces is complex linear.

Given a tangent vector $v\in T_mM$, the directional derivatives $\partial_v(F_+)$ and $\partial_v(F_-)$ are characterised by 
\begin{align*}
&	\partial_v F_+(z) = \partial_v a_0 + \partial_v a_1 z + \partial_v a_2 z^2 + \ldots
\\&	\partial_v F_-(z) = z^{-1} + \partial_v b_1 z + \partial_v b_2 z^2 + \ldots
\\&	dF_+ \circ \partial_v \varphi_+ + \partial_v F_+ \circ \varphi_+ = dF_- \circ \partial_v \varphi_- + \partial_v F_- \circ \varphi_-
\end{align*}
(where $\partial_v \varphi_+$ and $\partial_v \varphi_-$ are the directional derivatives of $\varphi_+$ and $\varphi_-$).
Since $i\partial_vF_+$ and $i\partial_vF_-$ satisfy the equations characterizing $\partial_{iv} F_+$ and $\partial_{iv} F_-$, we have $i\partial_vF_\pm=\partial_{iv}F_\pm$, as desired.

It remains to identify the welded sheaf with the sheaf $\cO_{\CC P^1\times M}$ of holomorphic functions on $\CC P^1\times M$.
Once again, this follows from the general fact that if a function is continuous everywhere, and holomorphic away from a real hypersurface, then it is holomorphic everywhere.

Finally, the map $F_+\circ\phi_+:S^1\times M\to \bbC P^1\times M$ 
in \eqref{eq: moreover the map...} is holomorphic in the second variable because 
$F_+(\phi_+(\theta,m))=F_+(\varphi_+^m(\theta))$,
and both $F_+$ and $\varphi_+$ are holomorphic.
\end{proof}

\subsection{Conformal welding of annuli}
\label{sec: conformal welding annuli}

Recall the definitions of an \emph{annulus} and a \emph{disc} --- Definitions~\ref{def: annulus} and~\ref{def: univ Teich space}.
Our next goal is to check that the operations $\cup:\Ann\times\Ann\to \Ann$, and $\cup:\Ann\times\,\cT\to \cT$ are well defined, where $\cT$ is universal Teichm\"uller space.

\begin{prop}\label{prop: conformal welding of annuli}
Let $A_1, A_2$ be annuli with boundary parameterizations $\varphi_{in/out}^i:S^1 \to \partial_{in/out} A_i$, and 
let $\phi:=\varphi_{out}^2 \circ (\varphi_{in}^1)^{-1}:\partial_{in}A_1\to \partial_{out}A_2$.
Then $A_1\cup_\phi A_2$, with the sheaf of rings $\cO_{A_1\cup_\phi A_2}$ given by
\begin{equation}\label{eq: glued sheaf}
\cO_{A_1\cup_\phi A_2}(U):=\big\{f:U\to \CC\,\big|\,f|_{U\cap A_i}\in\cO_{A_i}(U\cap A_i),\,\text{\small for}\;i=1,2\big\}
\end{equation}
is again an annulus.

When equipped with the boundary parameterizations
$\varphi_{in}^2:S^1\to \partial_{in}A_2=\partial_{in}(A_1\cup_\phi A_2)$ and $\varphi_{out}^1:S^1\to \partial_{out}A_1=\partial_{out}(A_1\cup_\phi A_2)$,
it defines an element of $\Ann$ denoted $A_1\cup A_2$
\end{prop}

\begin{proof}
Write $A_1$ as $D_1\setminus \mathring D_1'$ for some discs $D_1'\subset D_1\subset \bbC$.
Similarly, write $A_2=D_2\setminus \mathring D_2'$ for $D_2'\subset D_2\subset \bbC P^1$, where we now insist that $\infty \in D_2'$.
By Theorem~\ref{thm: conformal welding of discs}, we may pick an isomorphism $\psi:D_1\cup_{\phi} D_2\to \bbC\P^1$, uniquely specified by requiring that $\psi(\infty)=\infty$, and that $\psi(z)=z+o(1)$ around $\infty \in D_2$.
The image of $\partial_{in}A_1$ under $\psi$ (equivalently, the image of $\partial_{out}A_2$) is a smoothly embedded curve.

As a consequence of the Riemann mapping theorem for simply connected domains with smooth boundary \cite[Thm. 8.2]{Bell92}, the image of $\partial_{out}A_1$ under $\psi$ is also smooth, and so is the image of $\partial_{in}A_2$:
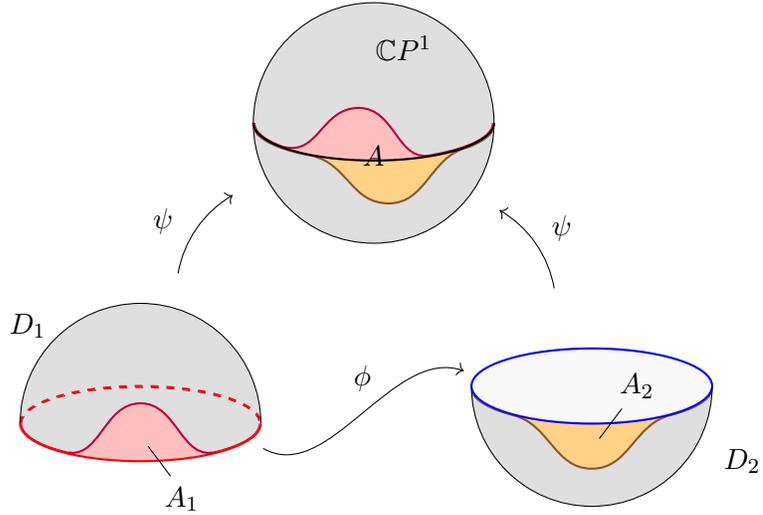
\[
\tikz{
\fill[gray!25] (-1.6,0) arc(-180:0:1.6 and .5) arc(0:180:1.6);
\fill[gray!25] (4.4,.5) arc(-180:0:1.6 and .5) arc(0:-180:1.6);
\draw[purple, dash pattern=on 2.95pt off 2.98pt] (-1.58,0) arc(180:0:1.58 and .48);
\draw[purple] (-1.58,0) arc(-180:-130:1.58 and .48) coordinate (A);
\draw[purple] (1.58,0) arc(0:-50:1.58 and .48) coordinate (B);
\fill[pink] ($(A)+(-0.002,-0.006)$) .. controls +(-16:.5) and +(180:.5) ..  (0,.27) .. controls +(0:.5) and +(180+16:.5) .. ($(B)+(0.002,-0.006)$) arc(-50:-130:1.58 and .48);
\draw[thick, purple] ($(A)+(-0.002,-0.006)$) .. controls +(-16:.5) and +(180:.5) ..  (0,.27) .. controls +(0:.5) and +(180+16:.5) .. ($(B)+(0.002,-0.006)$);
\draw[thick, red] (-1.6,0) arc(-180:0:1.6 and .5);
\draw[thick, dashed, red] (-1.6,0) arc(180:0:1.6 and .5);
\draw (-1.6,0) arc(180:0:1.6);
\draw[brown!70!black] (4.4,.43) arc(-175:-130:1.58 and .48) coordinate (A);
\draw[brown!70!black] (7.6,.43) arc(-5:-50:1.58 and .48) coordinate (B);
\fill[pink!70!yellow] ($(A)+(-0.002,0.008)$) .. controls +(-14:.5) and +(180:.5) ..  (6,-.6) .. controls +(0:.5) and +(180+14:.5) .. ($(B)+(0.002,0.008)$) arc(-50:-130:1.58 and .48);
\draw[thick, brown!70!black] ($(A)+(-0.002,0.008)$) .. controls +(-14:.5) and +(180:.5) ..  (6,-.6) .. controls +(0:.5) and +(180+14:.5) .. ($(B)+(0.002,0.008)$);
\draw[thick, blue, fill=gray!5] (6,.5) circle (1.6 and .5);
\draw (4.4,.5) arc(-180:0:1.6);
\draw[fill=gray!25] (3.1,4) circle (1.6);
\draw[purple] (1.52,4) arc(-180:-140:1.58 and .48) coordinate (A);
\draw[purple] (4.68,4) arc(0:-60:1.58 and .48) coordinate (B);
\fill[pink] ($(A)+(-0.002,-0.006)$) .. controls +(-16:.5) and +(180:.5) ..  (2.9,4.2) .. controls +(0:.5) and +(180+16:.5) .. ($(B)+(0.002,-0.006)$) arc(-60:-140:1.58 and .48);
\draw[thick, purple] ($(A)+(-0.002,-0.006)$) .. controls +(-16:.5) and +(180:.5) ..  (2.9,4.2) .. controls +(0:.5) and +(180+16:.5) .. ($(B)+(0.002,-0.006)$);
\draw[brown!70!black] (-2.9,3.5) +(4.4,.43) arc(-175:-120:1.58 and .48) coordinate (A);
\draw[brown!70!black] (-2.9,3.5) +(7.6,.43) arc(-5:-40:1.58 and .48) coordinate (B);
\fill[pink!70!yellow] ($(A)+(-0.002,0.008)$) .. controls +(-14:.5) and +(180:.5) ..  ($(-2.9,3.5) + (6.2,-.57)$) .. controls +(0:.5) and +(180+14:.5) .. ($(B)+(0.002,0.008)$) arc(-40:-120:1.58 and .48);
\draw[thick, brown!70!black] ($(A)+(-0.002,0.008)$) .. controls +(-14:.5) and +(180:.5) ..  ($(-2.9,3.5) + (6.2,-.57)$) .. controls +(0:.5) and +(180+14:.5) .. ($(B)+(0.002,0.008)$);
\draw[thick] (1.5,4) arc(-180:0:1.6 and .5);
\draw[->] (1.63,-.33) .. controls +(-30:1) and +(160:1) .. (4.3,.7);
\node at (2.95,.6) {$\phi$};
\draw[->] (.5,2) arc(170:120:1.5);
\draw[->] (5.5,1.8) arc(10:60:1.5);
\node at (.3,2.7) {$\psi$};
\node at (5.6,2.6) {$\psi$};
\node at (3.5,5) {$\CC P^1$};
\node at (3.1,3.57) {$A$};
\node at (-1.5,1.3) {$D_1$};
\node at (.55,-1) {$A_1$};
\node at (6.6,.5) {$A_2$};
\node at (8,-.5) {$D_2$};
\draw (.4,-.7) -- +(-.3,.4);
\draw (6.4,.2) -- +(-.3,-.4);
}
\]
The space $A:=A_1\cup_\phi A_2$ can be therefore identified with the subset of $\bbC\bbP^1$ that lies between these two curves.

It remains to identify the sheaf \eqref{eq: glued sheaf} with the set of those functions on $A$ that are continuous, smooth on the boundaries, and holomorphic in the interior. The only non-trivial condition is the last one, which follows from Morera's theorem, as in the proof of Theorem \ref{thm: conformal welding of discs}.
\end{proof}

We also have conformal welding of a disc to an annulus, via a similar (but simpler) argument to that of Proposition~\ref{prop: conformal welding of annuli}:
\begin{prop}\label{prop: conformal welding of annuli with disc}
Let $A$ be an annulus and $D$ be a disc, both equipped with boundary parameterizations.
Then $(A \cup_{S^1} D,\cO_{A \cup_{S^1} D})$ is again disc.
\end{prop}

Propositions~\ref{prop: conformal welding of annuli} and~\ref{prop: conformal welding of annuli with disc} admit families versions, for both real and complex families (with the real families indexed by manifolds with corners):

\begin{prop}\label{prop: conformal welding of annuli --- families}
Let $M$ be either a real manifold with corners, or a complex manifold, and let
$\underline A_1,\underline A_2\to M$ be two (real/complex) families of annuli with parametrised boundary.
Then $\underline A_1\cup \underline A_2$, with its boundary components $\partial_{in}(A_1\cup A_2)=\partial_{in}A_2$ and $\partial_{out}(A_1\cup A_2)=\partial_{out}A_1$ given the natural parametrisations inherited from $A_2$ and $A_1$, and its structure sheaf $\cO_{\underline A_1\cup \underline A_2}$ given by
\begin{equation}\label{eq: glued sheaf families}
\cO_{\underline A_1\cup \underline A_2}(U):=\big\{f:U\to \CC\,\big|\,f|_{U\cap \underline A_i}\in\cO_{\underline A_i}(U\cap \underline A_i),\,\text{\small for}\;i=1,2\big\},
\end{equation}
is again a (real/complex) family of annuli with parametrised boundary.
\end{prop}

\begin{proof}
Write $A_i^m$ for the fiber of $\underline A_i$ at $m\in M$, and $\varphi^m_{i,in}$, $\varphi^m_{i,out}$ for the boundary parametrisations.

Working locally in $M$, we may assume that $\underline A_1$ and $\underline A_2$ are determined by (smooth/ holomorphic) maps $M\to s\oplus s$ (as in \eqref{eq: def: two discs mapping to CP^1}).
We then follow the structure of the proof of Propositions~\ref{prop: conformal welding of annuli},
with an application of Proposition~\ref{prop: welding discs in smooth families}/\ref{prop: welding discs in holomorphic families} to ensure that the uniformizing maps $\psi^m:A_1^m\cup A_2^m\to\bbC P^1$ depend smoothly/holomorphically on $m$.
This yields an embedding $\underline A_1 \cup \underline A_2\hookrightarrow\bbC \times M$ which maps the fiber $A_1^m \cup A_2^m$ to the annulus in $\bbC$ bound by the Jordan curves $\psi^m \circ \varphi^m_{1,out}$ and $\psi^m \circ \varphi^m_{2,in}$. 
These curves depend smoothly/holomorphically on $m$, and thus define a (real/complex) family of annuli.
\end{proof}

Similarly, we have:
\begin{prop}\label{prop: conformal welding of annuli with disc --- families}
Let $M$ be either a real manifold with corners, or a complex manifold.
Let $\underline A\to M$ be a (real/complex) family of annuli, and let $\underline D\to M$ be a (real/complex) family of discs, both with parametrised boundary.
Then $\underline A\cup \underline D$, equipped with the sheaf
\begin{equation*}
\cO_{\underline A\cup \underline D}(U):=\big\{f:U\to \CC\,\,\big|\,\,
f|_{U\cap \underline A}\in\cO_{\underline A}(U\cap \underline A),\,
f|_{U\cap \underline D}\in\cO_{\underline D}(U\cap \underline D)
\big\},
\end{equation*}
and the obvious projection to $M$, is again a (real/complex) family of discs. \hfill $\square$
\end{prop}

\begin{cor} \label{cor: Ann x Ann -> Ann is holomorphic}
The gluing maps
\begin{align*}
\cup&:\Ann \times \Ann \to \Ann
\\
\cup&:\Ann \times\, \cT \to \cT
\end{align*}
are morphisms of diffeological spaces for both the real diffeologies (based on manifolds with corners), and the complex diffeologies on $\Ann$ and $\cT$.
\end{cor}

We finish this section by checking that families of discs and families of annuli do not have automorphisms.
This ensures that $\Ann$ and $\cT$ are really just `spaces', as opposed to stacks:

\begin{prop}\label{prop Riemann mapping theorem in families}
Let\!\!\; $M$\!\!\; be\!\!\; either\!\!\; a\!\!\; real\!\!\; manifold\!\!\; with\!\!\; corners,\!\!\; or\!\!\; a\!\!\; complex\!\!\; manifold.
\begin{enumerate}
\item
Two (real/complex) families $\underline D^{1},\underline D^{2}\to M$ of discs with para\-metrized boundaries are isomorphic if and only if the corresponding maps $M\to \cT$ are equal. The isomorphism $\underline D^1 \cong \underline D^2$ is unique if it exists.
\item
Two (real/complex) families $\underline A^{1},\underline A^{2}\to M$ of annuli with parametrised boundaries are isomorphic if and only if the corresponding maps $M\to \Ann$ are equal. Once again, the isomorphism $\underline A^1 \cong \underline A^2$ is unique if it exists.
\end{enumerate}
\end{prop}
\begin{proof}
\emph{1.}
As in the previous proof, we may assume without loss of generality that for each $i\in\{1,2\}$ the family $\underline D^i$ is associated to a (smooth/holomorphic) family $f^i:M\times S^1\to \bbC$ of smooth Jordan curves.
Let $D^i_m$ be the closed disc bound by $f^i_m:=f^i(m,-)$. We then have
\[
\underline D^i\,=\,\{(m,z)\in M\times \bbC:z\in D^i_m\}.
\]

If $\underline D^1\cong \underline D^2$, then the corresponding maps $M\to \cT$ are clearly equal.
Conversely, if $D^1_m\cong D^2_m$, $\forall m\in M$, with isomorphisms $F_m:D^1_m \to D^2_m$, then
we must show that the maps $F_m$ assemble to an isomorphism
\[
F: \underline D^1\to \underline D^2
\]
which is (smooth and fiberwise holomorphic in the interior/holomorphic in the interior), and smooth all the way to the boundary, with everywhere non-zero derivative.
Indeed, by Cauchy's integral formula we have
\begin{align*}
F_m(z)
&=\tfrac{1}{2\pi i} \int_{\partial D^1_m} f^2_m \circ (f^1_m)^{-1}(w) / (w-z) dw\\
&= \tfrac{1}{2\pi i} \int_{S^1} f^2(\theta) / (f^1(m,\theta)-z) \partial_\theta f^1(m,\theta)d\theta,
\end{align*}
so $F$ is (fibrewise holomorphic/holomorphic) in the interior.

To show that $F$ is smooth all the way to the boundary with everywhere non-zero derivative, we work locally in $M$.
Choose $a\in\bbC$ such that $a\in\mathring D^1_m$ for all $m$ in some open set $U\subset M$. By the family version of the smooth-boundary Riemann mapping theorem (Theorem~\ref{thm Riemann mapping theorem in families}),
the normalizations
$F^1_m(0)=a$, 
$F^2_m(0)=F_m(a)$, 
$\partial_z F^1_m(0)\in \mathbb R_+$,
$\partial_z F^2_m(0)\in \mathbb R_+\partial_zF_m(a)$
uniquely define maps $F^1$ and $F^2$
\[
\begin{tikzpicture}
\node (A) at (0,0) {$\underline D^1|_U$};   \node (B) at (3,0) {$\underline D^2|_U$};   \node (C) at (1.5,1.3) {$\bbD\times U$};
\draw[->] (C) --node[above, pos=.6]{$\scriptstyle F^1$} (A);
\draw[->] (C) --node[above, pos=.6, xshift=4]{$\scriptstyle F^2$} (B);
\draw[->] (A) --node[above]{$\scriptstyle F|_U$} (B);
\end{tikzpicture}
\]
that fit into a commutative diagram of fiberwise isomorphisms.
Theorem \ref{thm Riemann mapping theorem in families} states that $F^1$ and $F^2$ are smooth all the way to the boundary with everywhere non-zero derivative. The same therefore holds for $F$.

To finish the proof, we note that if an isomorphism $\underline{D}^1 \cong \underline{D}^2$ exists, then it is necessarily unique,
as holomorphic maps between discs are determined by their boundary values.

\emph{2.}
The case of annuli is largely similar.
We assume without loss of generality that the families $\underline A^i$ consist of embedded annuli. 

If $\underline A^1\cong \underline A^2$, then the corresponding maps $M\to \Ann$ are clearly equal.
So let us assume that 
$A^1_m\cong A^2_m$, $\forall m\in M$, with unique isomorphisms $F_m:A^1_m \to A^2_m$ of parametrised annuli.
Our task is to show that the maps $F_m$ assemble to an isomorphism
\[
F: \underline A^1\to \underline A^2
\]
which is (smooth and fiberwise holomorphic in the interior/holomorphic in the interior), continuous, and smooth when restricted to 
$\partial_{in}\underline A^1$ and $\partial_{out}\underline A^1$ with everywhere non-zero derivative.

Let $\underline{\bbD} = \bbD \times M$ be the trivial family and let $\underline{D}^j = \underline{A}^j \cup \underline{\bbD}$.
These are (real/complex) families of discs by Proposition~\ref{prop: conformal welding of annuli with disc --- families}.
Moreover from the description of the structure sheaf of the fibres $A^j_m \cup \bbD$ in Proposition~\ref{prop: conformal welding of annuli with disc} we see that the maps $F_m \cup \mathrm{id}$ give isomorphisms $A^1_m \cup \bbD \to A^2_m \cup \bbD$.
By the first part of this Proposition, the map $F \cup \mathrm{id}$ is  an isomorphism $\underline{D}^1 \to \underline{D}^2$ of (real/complex) families.
It follows immediately that $F$ is  (smooth and fiberwise holomorphic in the interior/holomorphic in the interior), continuous, and smooth when restricted to $\partial_{out}\underline A^1$.
We may repeat the above argument, this time welding $\bbD_-=\{z\in \bbC:\abs{z} \ge 1 \} \cup \{\infty\}$ to $\partial_{out}A^j_m$, to establish that $F$ is smooth when restricted to $\partial_{in}\underline{A}^1$ as well, and we conclude that $F$ is an isomorphism of (real/complex) families.

Finally, once again, an isomorphism $\underline{A}^1 \cong \underline{A}^2$ is unique if it exists as holomorphic maps between annuli are determined by their boundary values.
\end{proof}

\section{Annuli as paths into the Lie algebra of vector fields on~\texorpdfstring{$S^1$}{}}

Let us write $\cX(S^1)$ for the Lie algebra of complexified vector fields with smooth coefficients, and $\cX_\bbR(S^1)$ for its real subalgebra consisting of vector fields tangent to $S^1$. 
The Lie group of diffeomorphisms of $S^1$ has $\cX_\bbR(S^1)$ as its associated Lie algebra (provided we equip $\cX_\bbR(S^1)$ with the opposite of the usual bracket of vector fields).

We identify the two models of $\{z\in \bbC:|z|=1\}$ and $\{\theta\in \bbR\}/\theta\sim \theta+2\pi$ of the standard circle $S^1$ by the map $z=e^{2\pi i \theta}$.
Let $\cX^{in} \subset \cX(S^1)$ be the cone of inwards pointing vector fields (possibly tangential to $S^1$):
\[
\cX^{in} = \{ f(\theta)\partial_\theta \in \cX(S^1) \mid \mathrm{Im}(f(\theta))\ge 0\}
\]

\begin{defn}
A \emph{path in $\cX^{in}$} is a smooth map $X:[0,t_0] \to \cX^{in}$, for some $t_0>0$.
We write $\cP$ for the set of all path in $\cX^{in}$.

Such a path is said to have \emph{sitting instants} if there exits $\epsilon>0$ such that $X|_{[0,\epsilon]}$ and $X|_{[t_0-\epsilon, t_0]}$ vanish.
The collection of all such paths, which we denote $\cP^{si}\subset \cP$, is a semigroup under the operation of path concatenation
\end{defn}

If $M$ is a smooth manifold with boundary or with corners,
and $A$ is an annulus, then by a \emph{smooth map} $M\to A$ we shall mean a map of ringed spaces $(M,\cC^\infty_{M})\to (A,\cO_A)$, with $\cO_A$ as in \eqref{eq: def of O_A}. If the annulus is embedded in $\bbC$ then, by Lemma~\ref{lem: alternative description of O_A(A)}, this agrees with the notion of a smooth map $M\to \bbC$ whose image lands in $A$.

\begin{defn}\label{def: framing of A}
Let $A$ be an annulus with parametrised boundary.
A \emph{framing} of $A$ is a smooth surjective map
\[
h:S^1\times[0,t_0]\to A
\]
for some $t_0>0$, whose restriction to each $S^1\times\{t\}$ is an embedding, whose Jacobian determinant is non-negative, and that is compatible with the boundary parametrisations in the sense that $h|_{S^1\times\{0\}}=\varphi_{in}$ and $h|_{S^1\times\{t_0\}}=\varphi_{out}$.

A framing has \emph{sitting instants} if there exists $\varepsilon>0$ such that for all $\theta\in S^1$ the functions
$h|_{\{\theta\}\times [0,\varepsilon]}$ and 
$h|_{\{\theta\}\times [t_0-\varepsilon,t_0]}$ are constant.
\end{defn}

\begin{remark}
\label{rem: framings with sitting instants}
Any framing can be modified so as to introduce sitting instants by precomposing with a suitable map from $S^1\times [0,t]$ to itself.
\end{remark}

Let us identify $\cX(S^1)$ with the set of smooth functions on $S^1$ via the correspondence $f(\theta) \leftrightarrow f(\theta) \partial_\theta$.

\begin{defn}
The path $X\in\cP$ associated to a framing $h$ is given by
\begin{align*}
X:[0,t_0] \to \cX^{in},\quad
X(\theta,t) := 
-\frac{\partial h/\partial t}{\partial h/\partial \theta}.
\end{align*}
(Another notation we shall use is $-h_t/h_\theta$.)
If an annulus $A$ admits a framing $h$ such that $X(\theta,t)=-\frac{\partial h/\partial t}{\partial h/\partial \theta}$, then we write
\begin{equation}
\label{eq: Prod}
A = \prod_{t_0\ge \tau\ge 0} \Exp\big(X(\tau)d\tau\big).
\end{equation}
\end{defn}

Th relation \eqref{eq: Prod} is analog for $\Ann$ of the relation \eqref{eq: path in G}
between elements of a Lie group, and paths into the associated Lie algebra.
Our next goal is to provide a method for reconstructing the annulus $A$ from the path $X$ in \eqref{eq: Prod}:

\begin{defn}
Given a path $X\in \cP$, we say that a function
$f:S^1\times[0,t_0]\to \bbC$ is \emph{$X$-holomorphic} if it is smooth and satisfies
\begin{equation}\label{eq: def X-holom}
\partial f /\partial t = -X(\theta,t)\cdot \partial f/\partial \theta.
\end{equation}
Note that $X$-holomorphic functions are closed under multiplication, and thus form an algebra.
\end{defn}

\begin{remark}\label{rem: degenerate Beltrami}
If we set $z=t+i \theta$, a straightforward calculation 
shows that the condition \eqref{eq: def X-holom} is equivalent to the degenerate Beltrami equation
\[
\partial_{\overline{z}} f = \mu \partial_z f,
\]
where $\mu =\frac{X-i}{X+i}$.
The condition $\Im(X)\ge 0$ is equivalent to the condition $\mu(z)\in\{z:|z|\le 1\}\setminus\{1\}$.

We conjecture that such degenerate Beltrami equation always admit solutions. (But, as far as we know,
this is an open problem.)
\end{remark}

Our next main goal is the follwing theorem:

\begin{thm}\label{thm: X hol iff hol}
Let $h:S^1 \times [0,1] \to A$ be a framed annulus, and let $X = -h_t/ h_\theta$ be the corresponding path in $\cX^{in}$.

Then a function $f:S^1 \times [0,1] \to \bbC$ is $X$-holomorphic if and only if it is of the form $f = g \circ h$ for $g:A \to \bbC$ a function
which is continuous, smooth on the boundary, and holomorphic in the interior of $A$.
\end{thm}

\begin{proof}
Pick an embedding of $A$ into the complex plane.
We first check that for any function $g:A \to \bbC$ 
which is continuous, smooth on the boundary, and holomorphic in the interior of $A$,
the composite $f=g\circ h$ is $X$-holomorphic.
The function $g \circ h$ is smooth Lemma~\ref{lem: alternative description of O_A(A)}.
It is $X$-holomorphic since, upon decomposing $h(\theta,t)$ into its real and imaginary parts $h=a+ib$, we have:
\begin{align}\label{eqn: g circ h sub t}
(g \circ h)_t &= (g_x \circ h)a_t + (g_y \circ h)b_t \,= (g_z \circ h) \cdot (a_t + i b_t) \,= (g_z \circ h) \cdot h_t
\\
\notag
(g \circ h)_\theta &= (g_x \circ h)a_\theta + (g_y \circ h)b_\theta = (g_z \circ h) \cdot (a_\theta + i b_\theta) = (g_z \circ h) \cdot h_\theta,
\end{align}
from which \eqref{eq: def X-holom} readily follows.

If $f$ is $X$-holomorphic, then a function $g$ such $f = g \circ h$ will be furnished in Lemma~\ref{lem: X holomorphic functions factor through h}, where it will be established that $g$ is continuous and smooth on the boundary of $A$.
Finally, in Lemma~\ref{lem: Cauchy int formula for X holomorphic}, we will establish the formula
\[
g(z) = \tfrac{1}{2 \pi i} \int_{\partial A} \frac{g(w)}{z-w} \, dw
\]
for $z \in \mathring{A}$, which is visibly holomorphic in $z$.
\end{proof}

We now establish the lemmas that were used in the  above proof.

\begin{lem}\label{lem: X holomorphic functions factor through h}
Let $h:S^1 \times [0,1] \to A$ be a framed annulus, let $X = -h_t/ h_\theta$ be the associtated path, and let $f:S^1 \times [0,1] \to \bbC$ be an $X$-holomorphic function.
Then there exists a unique function $g:A \to \bbC$ such that $f = g \circ h$.
The function $g$ is continuous, and has smooth restrictions to $\partial_{in} A$ and to $\partial_{out} A$.
\end{lem}

\begin{proof}
The map $h$ is surjective between compact spaces. It is therefore a quotient map. The existence (and uniqueness) of a continuous function $g:A \to \bbC$ such that $f = g \circ h$ is therefore equivalent to $f$ being constant on the fibers $h^{-1}(z)$.

Fix $z\in A$, and fix an embedding $A\hookrightarrow \bbC$.
Let $L\subset \bbC$ be the straight line through $z$ with direction $\partial h/\partial \theta$, and let $\pi$ be the orthogonal projection onto $L$. By the implicit function theorem, $(\pi \circ h)^{-1}(z)$ is a smooth manifold in a neighbourhood of $h^{-1}(z)$. The curve $h^{-1}(z)$ is a closed subset of that smooth manifold, and is therefore itself smooth. The function $f$ is constant on $h^{-1}(z)$ because its directional derivative vanishes along that curve.

We may thus write $f = g \circ h$ for some continuous function $g$.
The restrictions $g|_{\partial_{in}A}$ and $g|_{\partial_{out}A}$ are smooth, since $f$ is smooth on the boundary and $h$ is assumed to restrict to a diffeomorphism on each boundary circle.
\end{proof}

\begin{lem}\label{lem: Cauchy int formula for X holomorphic}
Let $h:S^1 \times [0,1] \to A\subset \bbC$ be a framed embedded annulus,
let $X = -h_t/ h_\theta$, and let $f:S^1 \times [0,1] \to \bbC$ be an $X$-holomorphic function.
Write $f = g \circ h$ as in Lemma \ref{lem: X holomorphic functions factor through h}.
Then, for every $z \in \mathring{A}$, we have
\begin{equation}\label{eq: g(z) = int... (Cauchy int formula)}
g(z) = \tfrac{1}{2 \pi i} \int_{\partial A} \frac{g(w)}{z-w} \, dw
\end{equation}
where $\int_{\partial A}$ means $\int_{\partial_{out} A} - \int_{\partial_{in} A}$.
\end{lem}

\begin{proof}
Choose $(\theta_0,t_0) \in S^1 \times [0,1]$ such that $h(\theta_0, t_0) = z$.
Changing variables, formula \eqref{eq: g(z) = int... (Cauchy int formula)} is equivalent to
\begin{equation}\label{eq: f(theta, t) = int... (Cauchy int formula)}
f(\theta_0, t_0) = \tfrac{1}{2 \pi i} \int_{\partial A} \frac{f(h^{-1}(w))}{z-w} \, dw
=\int_{\partial(S^1 \times [0,1])} \frac{f(\theta,t)}{z-h(\theta,t)}\, d(h(\theta,t)).
\end{equation}
Here, we have used that $h$ restricts to a diffeomorphism on each of the boundary components $S^1 \times \{0\}$ and $S^1 \times \{1\}$, and $h^{-1}$ refers to the inverse of these diffeomorphisms.

Let $C_\varepsilon$  be the circle of radius $\varepsilon$ around $z$, and let $\Sigma_\varepsilon$ be its preimage under $h$. There is a dense set of $\varepsilon$’s for which $\Sigma_\varepsilon$ is smooth (as the function `distance to $z$' $\!\circ\, h$:
$S^1 \times [0,1] \xrightarrow{h} A \rightarrow \bbR$
is smooth, and hence admits a dense set of regular values). 
For any such $\varepsilon$, $h$ provides a smooth map of $\Sigma_\varepsilon \twoheadrightarrow C_\varepsilon$.

We claim that for any such $\varepsilon$, we have
\begin{equation}
\label{eq: integral over C_eps}
  \int_{\partial(S^1 \times [0,1])}     \frac{f(\theta,t)}{z-h(\theta,t)}\, d(h(\theta,t))
=\int_{\partial\Sigma_\varepsilon} \frac{f(\theta,t)}{z-h(\theta,t)}\, d(h(\theta,t)).
\end{equation}
Indeed, by Stokes' Theorem, it suffices to show that the $1$-form $k \, dh$ is closed on the region bounded by $\partial(S^1 \times [0,1]) \cup \Sigma_\varepsilon$, where 
$
k(\theta,t) := \tfrac{f(\theta,t)}{z-h(\theta,t)}
$.
The function $k$ is $X$-holomorphic because $X$-holomorphic functions form an algebra and $(\theta,t) \mapsto (z-h(\theta,t))^{-1}$ is the composite of $h$ and a holomorphic function.
We then calculate:
$d(k \, dh) = dk \, dh 
= h_\theta (k_\theta d\theta + k_t dt)(d \theta - X \, dt)
= -h_\theta (X k_\theta + k_t) d\theta \, dt = 0$,
which establishes \eqref{eq: integral over C_eps}.

We are reduced to showing that the follows equation holds:
\[
f(\theta_0, t_0) \,=\, \tfrac{1}{2 \pi i} \int_{\Sigma_\varepsilon} \frac{f(\theta, t)}{z-h(\theta, t)} d(h(\theta, t)).
\]
The right hand side is independent of $\varepsilon$ (provided $\Sigma_\varepsilon$ is smooth), so is enough to prove that
\[
f(\theta_0, t_0) \,=\, \tfrac{1}{2 \pi i}\,\,\lim_{\varepsilon \to 0}  \int_{\Sigma_\varepsilon} \frac{f(\theta, t)}{z-h(\theta, t)} d(h(\theta, t)).
\]
Recall that $f = g \circ h$. Changing variables one last time, we can rewrite our desired formula as
\[
g(z) \,=\, \tfrac{1}{2 \pi i}\,\,\lim_{\varepsilon \to 0} \int_{C_\varepsilon} g(w)(z - w)^{-1}  dw.
\]
This holds true for any continuous function $g$.
\end{proof}

\begin{cor}\label{cor: framed annuli are iso iff X equals tilde X}
Let $h:S^1 \times [0,1] \to A$, and $\tilde h: S^1 \times [0,1] \to \tilde A$ be framed annuli, embedded in the complex plane.
Let $X=-h_t/h_\theta$ and $\tilde X = -\tilde h_t / \tilde h_\theta$ be the corresponding paths in $\cX^{in}$. Then the following are equivalent:
\begin{enumerate}
\item There exists a (necessarily unique) isomorphism $\psi: A \to \tilde A$ such that $\psi \circ h = \tilde h$
\item $X = \tilde X$
\end{enumerate}
\end{cor}
\begin{proof}
If $\psi: A \to \tilde A$ is such that $\psi \circ h = \tilde h$ then, computing as in \eqref{eqn: g circ h sub t}, we have
\[
\tilde X = -\frac{\tilde h_t}{\tilde h_\theta} = -\frac{(\psi' \circ h)\cdot  h_t}{(\psi' \circ h)\cdot  h_\theta} = X.
\]

Conversely, suppose that $X = \tilde X$.
By construction, $h$ is $X$-holomorphic, and $\tilde h$ is $\tilde X$-holomorphic.
Hence $h$ and $\tilde h$ are both $X$-holomorphic and $\tilde X$-holomorphic.
By Theorem~\ref{thm: X hol iff hol}, there exist holomorphic functions $\psi$ and $\tilde \psi$, defined on $A$ and $\tilde A$ respectively, such that $\tilde h = \psi \circ h$ and $h = \tilde \psi \circ \tilde h$.
It folllows that $h = \tilde \psi \circ \psi \circ h$, and $\tilde h = \psi \circ \tilde \psi \circ \tilde h$.
Finally, since $h$ and $\tilde h$ are surjective, we get $\tilde \psi \circ \psi = \id_A$ and $\psi \circ \tilde \psi = \id_{\tilde A}$, so
$\psi$ is an isomorphism.
At last, since $h$ and $\tilde h$ are surjective, such an isomorphism $\psi$ is unique.
\end{proof}

Recall that given a path $X:[0,t_0] \to \cX^{in}$, we write
\begin{equation}
\label{eq: ljnksdfjnb}
A = \prod_{t_0\ge \tau\ge 0} \Exp\big(X(\tau)d\tau\big)
\end{equation}
if $A$ is an annulus which admits a framing $h$ such that $X=-\frac{h_t}{h_\theta}$. 
In light of Corollary~\ref{cor: framed annuli are iso iff X equals tilde X}, such an annulus is unique up to isomorphism (if it exists). So the formula \eqref{eq: ljnksdfjnb} describes a well-defined element of $\Ann$, provided it exists. 

\begin{defn}
We say that a path $X$ is \emph{geometrically exponentiable}
if there exists an annulus $A$ and a framing $h$ of $A$ such that $X=-\frac{h_t}{h_\theta}$.
That is, a path $X$ is geometrically exponentiable if $\prod_{t_0\ge \tau\ge 0} \Exp\big(X(\tau)d\tau\big)$ exists.

We write $\cP^{ge} \subseteq \cP$ for the collection of geometrically exponentiable paths, and $\cP^{ge,si} := \cP^{ge} \cap \cP^{si}$ for those that furthermore admit sitting instants.
\end{defn}

\begin{remark}
\label{rem: conjecture that all paths are geometrically exponentiable}
We conjecture that all paths are geometrically exponentiable -- see Remark~\ref{rem: degenerate Beltrami}.
Given a path $X$ that comes from an annulus by means of a framing, we can recover the annulus as the max-spectrum (set of maximal ideals) of the ring of $X$-holomorphic functions.
We conjecture that for any path $X\in \cP$, the max-spectrum of the ring of $X$-holomorphic functions is an annulus (Definition~\ref{def: annulus}).
\end{remark}

If $h_1$ and $h_2$ are framings of $A_1$ and $A_2$, and if both $h_1$ and $h_2$ have sitting instants, then their concatenation is a framing of $A_1\cup A_2$. It follows that the set $\cP^{ge,si}$ of geometrically exponentiable paths with sitting instants is a semigroup, and that the exponential map $\cP^{ge,si}\to \Ann$ is a semigroup homomorphism.

Our next goal is to show that this exponential map is surjective, i.e., that every annulus admits a framing.
It will be technically convenient to first construct framings that are compatible with one of the two boundary parametrisations, but not the other:

\begin{defn}
A \emph{semi-framing} of an annulus $A$ is
a surjective map $h:S^1\times[0,1]\to A$ whose restriction to each $S^1\times\{t\}$ is an embedding, whose Jacobian determinant is non-negative, and whose restriction to $S^1 \times \{0\}$ agrees with the parametrisation of $\partial_{in} A$.
We write $\Ann^{\mathit{sf}}$ for the space of annuli equipped with a semi-framing.
\end{defn}



\begin{prop}\label{prop: exist semi-framings}
Every annulus admits a semi-framing.
\end{prop}

\begin{proof}
Let $A$ be an annulus.
Our strategy will be to construct a 1-dimensional foliation $\cL$ of $A$ whose leaves are either intervals that meet the boundary of $A$ transversely, or single points. The desired semi-framing $S^1\times[0,1]\to A$ is then provided by the arc-length parametrisation of the leaves of $\cL$ (after identifying $\partial_{in} A$ with $S^1$ using the given parametrisation).

Let $\bbS_A := (T_A-\{\text{zero section}\})/\bbR_+$ be the unit tangent bundle of $A$ (the tangent bundle $T_A$ is defined in Definition~\ref{def: tangent bundle of annulus}). We will construct a distribution $H:A\to \bbS_A$, and define $\cL$ to be the foliation tangent to that distribution.
The distribution $H$ is itself built by patching together 
two distribution $H^1$ and $H^2$ using a partition of unity, where $H^1$ is defined in the complement $U:=A\setminus A_{\text{thin}}$ of $A_{\text{thin}}=\partial_{in}A\cap\partial_{out}A$, and $H^2$ is defined in some neighbourhood $V$ of $A_{\text{thin}}$.

We postpone the definitions of $H^1$ and $H^2$ to describe the patching procedure.
Let $\tilde \bbS_A$ be the fiberwise universal cover of $\bbS_A$, and let $\tilde H^i:A\to \tilde \bbS_A$ be lifts of $H^i$.
Pick functions $\lambda,\mu:A\to [0,1]$ with
$\supp(\lambda)\subset U$ and $\supp(\mu)\subset V$ that form a partition of unity.
Since $\tilde \bbS_A$ is an $\bbR$-principal bundle, it makes sense to consider $\tilde H:=\lambda \tilde H^1+ \mu \tilde H^2$.\vspace{-1mm}
We let $H$ be the composite $A\xrightarrow{\scriptscriptstyle \tilde H} \tilde \bbS_A \twoheadrightarrow \bbS_A$.

It remain to define $H^1$ and $H^2$. Let $f:A-A_{\text{thin}}\to \bbR$ be the solution of the Dirichlet problem with boundary conditions $f|_{\partial_{in}(A)}=0$ and $f|_{\partial_{out}(A)}=1$. We then let $H^1$ be the distribution orthogonal to the level curves of $f$.
To define $H^2$, identify $A\cup \bbD$ with $\bbD$ via the unique holomorphic map that sends $0\in \bbD\subset A\cup \bbD$ to $0\in\bbD$, and has positive derivative at that point.
The distribution $H^2$ is the one tangent to the foliation by straight lines through the origin of $\bbC$, and the open set $V$ is the complement of the locus where $H^2$ meets $\partial A$ non-transversely.
\end{proof}

The above construction does not quite work in families because we don't know whether the required solutions to the Dirichlet problem depend smoothly on the annulus (note that the topology of the interior of the annulus may change throughout a family).
We supply a modified argument which works in families.
Our argument in fact works over the universal family of annuli (whose parameter space is $\Ann$ itself).

Given an annulus $A\in \Ann$, there is a unique uniformizing map $\bbD \cup A\to \bbD$ that sends $0$ to $0$, and has positive derivative at that point. This identifies $\Ann$ with the product $\Univ_0 \times \Diff(S^1)$,
where $\Univ_0:=\{f\in \Univ \,|\, f(0)=0, f'(0)>0\}$.
In that model, the total space of the universal family of annuli is given by 
\[
\underline A := \{(z,f,\gamma)\in \bbC \times \Univ_0 \times \Diff(S^1) \,|\, z \in A_{f,\gamma}\}
\]
where $A_{f,\gamma}=\bbD \setminus f(\mathring{\bbD})$.

\begin{prop}\label{prop: exist semi-framings in families}
There exists an assignment of a semi-framing $h_A: S^1 \times [0,1] \to A$ to every annulus $A\in \Ann$ such that the map
\begin{align*}
(S^1 \times [0,1]) \times \Ann \,\,&\to\,\quad {\underline A}
\\
\big((\theta,t),A\big)\qquad &\mapsto\,\, h_{A}(\theta,t)
\end{align*}
is continuous, and smooth in the sense of diffeologies.
\end{prop}

\begin{proof}
Almost all the steps of the proof of Proposition~\ref{prop: exist semi-framings} can be applied verbatim to the case of the universal family $\underline A$, with two exceptions that requires some care: the existence of smooth partitions of unity, and the smoothness in families of solutions to the Dirichlet problem.

With regard to the existence of smooth partitions of unitary, the two sets that need to be separated by a smooth function are
$\underline A_{\text{thin}}$, and the subset of $\partial \underline A$ which is not transverse to $H^2$.
These are disjoint closed subsets $K_1,K_2\subset \underline A$, and $\underline A$ is itself a subset of a Fr\'echet space isomorphic to the space $s$ of rapidly decreasing sequences. The closures of $K_1$ and $K_2$ inside the Fr\'echet space remain disjoint, and the existence of a smooth partition of unity follows from \cite[Thm 16.10]{KrieglMichor1997}.

In order to ensure the smoothness in families of solutions to the Dirichlet problem, we modify the domain on which the Dirichlet problem is formulated so as to eliminate thin parts.
Let $\lambda,\mu:\underline A\to [0,1]$ be the two functions forming the partition of unity, with $\mu=1$ on a neighbourhood $W$ of $\underline A_{\text{thin}}$.
Let $\nu:\partial_{out}\underline A\to \bbR_{\ge 0}$ be a smooth function
satisfying $\underline A_{\text{thin}}\subset \supp(\nu)\subset W$.
For any $(f,\gamma)\in \Univ_0 \times \Diff(S^1)\cong \Ann$, we may then solve the Dirichlet problem on the domain
\begin{equation}
\label{fix Dirichlet}
A_{f,\gamma} \cup \{z\in \bbC \,:\, 1\le |z|\le 1+\nu(z/|z|,f,\gamma) \}
\end{equation}
instead of $A_{f,\gamma}$.
Since \eqref{fix Dirichlet} has no thin parts, the solution of the Dirichlet problem depends smoothly on the parameters $f$ and $\gamma$.
\end{proof}

If follows that:

\begin{cor}\label{cor: semi annuli contractible}
The space $\Ann^{\mathit{sf}}$ of annuli equipped with a semi-framing is contractible.
\end{cor}
\begin{proof}
By Proposition~\ref{prop: exist semi-framings in families} we may choose a semi-framing for each annulus $A\in \Ann^{\mathit{sf}}$, in a way that depends smoothly on the annulus.
A semi-framing $h:S^1\times[0,1]\to A$ provides a path $\{h|_{S^1\times[0,t]}\}_{t\in[0,1]}$ in $\Ann^{\mathit{sf}}$ from $A$ to the unique totally thin annulus.
Those paths assemble to a continuous, indeed smooth, contracting homotopy $\Ann^{\mathit{sf}} \times [0,1] \to \Ann^{\mathit{sf}}$.
\end{proof}

\begin{cor}\label{cor: exist framings}
Let $A \in \Ann$.
Then there exists a framing $h:S^1 \times [0,1] \to A$.
\end{cor}
\begin{proof}
By Proposition~\ref{prop: exist semi-framings}, there exists a semi-framing $h:S^1 \times [0,1] \to A$.
To adjust the parametrisation of the outgoing boundary, pick a path $\psi:[0,1]\to\Diff(S^1)$ connecting
\[
\psi_0:=(h|_{S^1 \times \{0\}})^{-1}\circ \phi_{in}
\qquad\text{and}\qquad
\psi_1:=(h|_{S^1 \times \{1\}})^{-1}\circ \phi_{out},
\]
and let $\Psi(\theta,t):=(\psi(t)(\theta),t)$.
The map $h\circ\Psi:S^1 \times [0,1] \to A$ is a framing.
\end{proof}

\begin{lem}
If $A$ is an annulus, then the space of framings has $\bbZ$ many path-connected components, indexed by their `winding number' (more precisely, $\pi_0$ of the space of framings is a torsor for $\bbZ$). Moreover, each connected component is contractible.
\end{lem}

\begin{proof} Let $G\subset \Diff(S^1)^{[0,1]}$ be the subgroup of smooth paths $[0,1]\to \Diff(S^1)$ that begin and end at $\id_{S^1}$.
The group $G$ is homotopy equivalent to $\bbZ$, and acts freely and properly on the space of framings of $A$. 
Let us call the quotient of the space of framings by the action of $G$ the space of `horizontal foliations' of $A$.
To prove our result, it suffices to show that the space of horizontal foliation of $A$ is contractible.

Let $\cJ$ denote the space of unparameterised Jordan curves in $A$.
We may identify horizontal foliations of $A$ with certain paths $[0,1]\to \cJ$ connecting $\partial_{in}A$ and $\partial_{out}A$.

Let $f:[0,1] \to \cJ$ be the horizontal foliation provided by Corollary~\ref{cor: exist framings}.
Given another horizontal foliation $g:[0,1] \to \cJ$, represented by some framing $h:S^1\times[0,1]\to A$,
let us write $A_{[x,y]}\subset A$ for the image of $S^1\times[x,y]$ under $h$,
and let us write $g|_{[x,y]}$ for the horizontal foliation of $A_{[x,y]}$ given by the restriction of $g$.

Applying Proposition~\ref{prop: exist semi-framings} to $A_{[0,t]}$, we get a horizontal foliation $k_t:[0,t]\to \cJ$ of $A_{[0,t]}$, which may furthermore be chosen to depend smoothly on $A$ and on $t$.
If we had included piecewise smooth maps into $\cJ$ in our space of horizontal foliations,
then the concatenation $k_t \cup g|_{[t,1]}$ of $k_t$ with the horizontal foliation $g|_{[t,1]}$ of $A_{[t,1]}$ would have been a contracting homotopy, with contraction point $f$, of the space of horizontal foliations.

Unfortunately, $k_t \cup g|_{[t,1]}$ is typically not smooth at $t$.
However, we may precompose the map $[0,1]^2 \to \cJ$ given by $(t,s) \mapsto (k_t \cup g|_{[t,1]})(s)$  with the map $\varphi$ from Lemma~\ref{lem: smoothing out the diagonal} below to obtain a smooth family of smooth horizontal foliations connecting $g$ and $f$.
In other words, a path from $g$ to $f$.
This path depends smoothly on $g$ (and $f$ doesn't depend on $g$), so this is a contracting homotopy of the space of horizontal foliations.
\end{proof}

\begin{lem}\label{lem: smoothing out the diagonal}
There exists a smooth surjective function
\[
\varphi(x,y) = (\varphi_1(x), \varphi_2(x,y)): [0,1]^2 \to [0,1]^2,
\]
such that 
$\varphi(0,y)=(0,y)$, $\varphi(1,y)=(1,y)$, $\partial_y \varphi_2(x,y)\ge 0$ and such that the following property holds:
for every smooth manifold $M$ and every piecewise smooth map $F:[0,1]^2 \to M$ which is smooth on $\{(x,y):x \le y\}$, and smooth on $\{(x,y):x \ge y\}$ the composite $F \circ \varphi$ is smooth.
\end{lem}

\begin{proof}
Let $f:\bbR\to \bbR$ be a smooth strictly increasing function such that $f$ and all of its derivatives vanish at $x=0$.
Set
\[
f_t(y) := \big( f(y-t) - f(-t) \big) \big/ \big( f(1-t) - f(-t) \big)
\]
and observe that $f_t(0)=0$, $f_t(1)=1$, and that all the derivatives of $f_t$ vanish at $y=t$.
Let $t:[0,1] \to [0,1]$ be a smooth weakly increasing function that sends $[0,1/3]$ to $0$ and $[2/3,1]$ to $1$.
We then let $\varphi_1(x) := f_{t(x)}(t(x))$, and
\[
\varphi_2(x,y) := \left\{ 
\begin{array}{cl} 
\lambda(x)f_0(y) + (1-\lambda(x))y & x \in [0,1/3]\\
f_{t(x)}(y) & x \in [1/3,2/3]\\
\mu(x)y + (1-\mu(x))f_1(y) & x \in [2/3, 1]
\end{array}
\right.
\]
where $\lambda$ increases from $0$ to $1$ over the interval $[0,1/3]$,
$\mu$ increases from $0$ to $1$ over the interval $[2/3,1]$, and both $\lambda$ and $\mu$ have
all derivatives that vanish at the endpoints.
The map $\varphi(x,y) = (\varphi_1(x), \varphi_2(x,y))$ has all the desired properties.
\end{proof}

\begin{prop}\label{prop: Exp is surjective and admits local sections}
The exponential map $\cP^{ge} \to \Ann$ given by
\[
X \mapsto \prod_{t_0\ge \tau\ge 0} \Exp\big(X(\tau)d\tau\big).
\]
is surjective and admits local sections.
\end{prop}
\begin{proof}
The surjectivity claim is the statement of Corollary~\ref{cor: exist framings}.
The existence of local sections is the fact that the proof of Corollary~\ref{cor: exist framings} works in families.
Specifically, semiframings may be chosen in families by Proposition~\ref{prop: exist semi-framings in families}.
It then remains to choose paths $[0,1]\to\Diff(S^1)$ that adjust the outgoing boundary parametrisations, which can also be done in families, but only locally in the parameter space.
\end{proof}

Recall that $\cX(S^1)$ denotes the Lie algebra of complexified vector fields on $S^1$.
Let $\cX^{univ}\subset \cX(S^1)$ be the closed span of $\{L_{-1}, L_0, L_1, L_2, \ldots \}$, and
let $\cP^{univ} \subset \cP$ be the space of paths whose image lies in $\cX_{in}^{univ}:= \cX^{univ} \cap \cX^{in}$.
We had left open the question of whether every path in $\cP$ is geometrically exponentiable (see Remark~\ref{rem: conjecture that all paths are geometrically exponentiable}). 
We record here a positive result in the special case of paths that take values in $\cX^{univ}$.

\begin{prop}
If $X \in \cP^{univ}$, then $\prod_{t_0\ge \tau\ge 0} \Exp\big(X(\tau)d\tau\big)$ exists and lies in $\Univ$.
In particular, $\cP^{univ} \subset \cP^{ge}$.
\end{prop}

\begin{proof}
We must construct an annulus $A$ and a framing $h:S^1 \times [0,1] \to A$ with $X=-\tfrac{h_t}{h_\theta}$.

Let $Y(z,t):=i z X(\log(z)/i,t)$, so that $Y(e^{i\theta},t)\partial_z=X(\theta,t)\partial_\theta$.
In terms of the new variable $z=e^{i \theta}$, the desired equation
$X(\theta,t) = -\tfrac{h_t}{h_\theta}$ becomes $Y(z,t) = -\tfrac{k_t}{k_z}$, where $k$ and $h$ are related by $k(z,t)=h(\log(z)/i,t)$.

Let $g(z,t,s)$ be the solution of the holomorphic flow equation 
\[
g_s(z,t,s) = Y(g(z,t,s),s),
\]
defined for $t \le s \le 1$, with initial value $g(z,t,t) = z$.
Note that since $X \in P[\cX^{in}]$, the vector fields $Y(z,t) \partial_z$ are inward pointing on the unit circle, and thus the flow exists for all times $s\in [t,1]$. Each map $z\mapsto g(z,t,s)$ is a univalent map $\bbD\to \bbD$.

We set $k(z,t) := g(z,t,1)$, and check that it satisfies the desired equation $Y(z,t) = -\tfrac{k_t}{k_z}$. In terms of the function $g$, this reads:
\[
g_t(z,t,1) = -g_z(z,t,1) Y(z,t).
\]

For that, we differentiate in $s$ the equation $g(g(z,t,s),s,1) = g(z,t,1)$ to get
\begin{align*}
g_z(g(z,t,s),s,1) &\cdot g_s(z,t,s) + g_t(g(z,t,s),s,1) = 0
\Rightarrow\\
g_t(g(z,t,s),s,1) &= - g_z(g(z,t,s),s,1)  Y(g(z,t,s),s).
\end{align*}
Here, $g_t$ denotes the differentiation of $g(\cdot,\cdot,\cdot)$ in the second input.
Setting $s = t$ yields $g_t(z,t,1) = - g_z(z,t,1) Y(z,t)$.

Finally, we define $A$ to be the image of $S^1 \times [0,1]$ under $k$, equivalently $A=\bbD\setminus k(\mathring\bbD,0)$. The latter description is visibly in $\Univ$.
\end{proof}

The previous proposition proved the existence of the exponential map $\cP^{univ} \to \Univ$.
We now show that it is surjective, and that it admits local sections.

\begin{prop} 
The exponential map $\cP^{univ} \to \Univ$ is surjective and admits local sections.
\end{prop}
\begin{proof}
Recall from Proposition \ref{prop: Exp is surjective and admits local sections} that the time-ordered exponential map $\cP^{ge} \to \Ann$ admits local sections.
Given some annulus $A_0 \in \Univ$, pick a local section defined on a neighborhood $U\subset \Univ$ of $A_0$ with values in $\cP^{ge}$.
For every annulus $A \in U$ we thus have a framing $h_A: S^1 \times [0,1] \to A$.
As before, we write $A_{[x,y]}$ for $h_A(S^1 \times [x,y])$.

Weld in a disc $\bbD \cup A$, and consider the family of discs $D_{A,t} \subset \bbD \cup A$ given by 
\[
D_{A,t} := \bbD \cup A_{[t,1]}.
\]
This is a smoothly varying family of discs by Proposition~\ref{prop: conformal welding of annuli with disc --- families}. 
Now apply the Riemann mapping theorem to each $D_{A,t}$, while fixing a value and a direction at $0 \in \bbD$.
In that way, we get a smooth family of maps $\varphi_{A,t} : \bbD \to \bbD\cup A$.
The image of $\varphi_{A,t}$ is the subset $D_{A,t}$ of $\bbD \cup A$.
From that, we get a family of elements of $\Univ$, namely 
\[
\psi_{A,t}:=\varphi_{A,0}^{-1} \circ \varphi_{A,t},
\]
and the path associated to this framing
\[
X_A(\log(z)/i,t):= \left.-\frac{\partial_t  \psi_{A,t}(z)}{i z \partial_z \psi_{A,t}(z)} \right|_{S^1 \times [0,1]} \in P[\cX^{in}_{\Univ}]
\]
is the value at $A\in U\subset \Univ$ of our desired section $U\to \cP^{univ}$.
\end{proof}

%
%
%
%
%

\section{Central extension of the semigroup of annuli}

The Virasoro Lie algebra
$\{\sum a_nL_n+kC : a_n, k \in \bbC, \text{ finitely many $a_n$ non-zero}\}$ is
spanned by $\{L_n\}_{n\in \mathbb Z}$ and $C$, and has commutation relations
\[
[L_m,L_n] = (m-n)L_{n+m} + \frac{C}{12} (m^3-m) \delta_{n+m,0}.
\]
Letting $\mathrm{Witt}$ be the Lie algebra of complexified vector fields $f(z)\partial_z$ on the circle, with $f$ a Laurent polynomial,
we then have a well-known short exact sequence (a central extension) of Lie algebras:
\[
0\to \bbC \to \Vir \to \mathrm{Witt} \to 0
\]
These algebras admit completions, where the sequence of coefficients $(a_n)$ is allowed to be rapidly decreasing (as opposed to having finite support), and $f$ is allowed to be smooth (as opposed to being a Laurent polynomial).
The completion of $\mathrm{Witt}$ is the Lie algebra previously denoted $\cX(S^1)$.

The goal of this section is to construction a central extension
\begin{equation}\label{eq: centr ext 0-C-tAnn-Ann-0}
0\to \bbC \times \bbZ \to \tAnn \to \Ann \to 0.
\end{equation}
of the semigroup of annuli by the group $\bbC\times\bbZ$ that corresponds to the Virasoro central extension of $\cX(S^1)$.

We will construct $\tAnn$ as a complex diffeological space, and we will show that \eqref{eq: centr ext 0-C-tAnn-Ann-0} is a central extension of diffeological semigroups, both for the real diffeology (based on manifolds with corners) and for the complex diffeology, in the following sense:

\begin{defn}\label{def: central extension of diffeological semigroups}
A \emph{central extension of diffeological semigroups} is a homomorphism of diffeological semigroups
\[
\tilde G \to G
\]
whose kernel $Z$ is a group, and which is a $Z$-principal bundle, meaning that for every smooth (holomorphic) map $M\to G$ from some finite dimensional real (complex) manifold $M$ and every point $m\in M$, there exists an open neighbourhood $U\subset M$ such that the pullback 
$\tilde G\times_G U \to U$
is $Z$-equivariantly isomorphic to a product $Z\times U \to U$.
\end{defn}

\noindent
We begin by constructing the semigroup $\tAnn$.

A technical annoyance of working with semigroups as opposed to groups is that the left-translation map $A \cup -:\Ann\to \Ann$ does not induce an isomorphism of tangent spaces (see Proposition~\ref{prop: tangent space of Ann}).
This makes it tricky to work with the notions of left-invariant vector field and left-invariant differential forms on $\Ann$.
To deal with that technical difficulty, we consider the following replacement: $\Ann^{\le \ast} := \bigsqcup_{A\in \Ann}\Ann^{\le A}$ (equipped with the disjoint union topology/diffeology), where 
\[
\Ann^{\le A}:= \{A_1 \in \Ann \,|\,  \exists A_2 : A_1 \cup A_2 = A\}.
\]
The space $\Ann^{\le A}$ can also be described as the space of embeddings $S^1\hookrightarrow A$ that split it into two annuli $A_1$ and $A_2$.
For any annulus $A_1\in \Ann^{\le A}$, the tangent space of $\Ann^{\le A}$ at $A_1$ is just $\cX(\partial_{in}A_1)\cong \cX(S^1)$.
In particular, the operation
\[
A_1 \cup - \,:\, \Ann^{\le A} \to \Ann^{\le A_1\cup A}
\]
does induce an isomorphism of tangent spaces.

\begin{lem}\label{Ann le A is homotopy equivalent to S^1}
For any annulus $A\in \Ann$, the space $\Ann^{\le A}$ is homotopy equivalent to a circle.
\end{lem}
\begin{proof}
The quotient of $\Ann^{\le A}$ by the action of $\Diff(S^1)$ which reparametrises the inner boundary is contractible by the same argument as in Corollary~\ref{cor: semi annuli contractible}.
The result follows since the above action is principal, and $\Diff(S^1)$ is homotopy equivalent to $S^1$.
\end{proof}

Using the trivialisation of the tangent bundle of $\Ann^{\le *}$, we define a vector field to be a smooth map (in the diffeological sense) from $\Ann^{\le *}$ to $\cX(S^1)$.
Similarly, a $k$-form is a smooth map from $\Ann^{\le *}$ to the continuous dual of the $k$-th exterior power of $\cX(S^1)$.



\begin{defn}
A vector field / differential form on $\Ann^{\le \ast}$ is called \emph{left invariant} if it is invariant under the two maps
\begin{gather*}
\Ann^{\le A}\hookrightarrow \Ann^{\le A\cup B}: A_1\mapsto A_1,\,\quad\text{and}
\\
\Ann^{\le B}\hookrightarrow \Ann^{\le A\cup B}: A_1\mapsto A\cup A_1
\end{gather*}
for all $A,B\in\Ann$ (these maps all induce isomorphisms of tangent spaces).
\end{defn}

Since $\Ann^{\le \ast}$ has trivial tangent bundle,
given an element of $\cX(S^1) = T_1\Ann$, we may consider the corresponding left-invariant vector field on $\Ann^{\le \ast}$. 

Let $\cJ \subset C^\infty(S^1, \bbC)$ be the space of parametrised Jordan curves. If $A$ is an annulus, then an embedding of $A$ into $\bbC$ provides an embedding
\begin{equation}\label{eq: emb into J}
\Ann^{\le A}\hookrightarrow \cJ\subset C^\infty(S^1, \bbC): A_1\mapsto \partial_{in}A_1
\end{equation}
In terms of the embedding \eqref{eq: emb into J}, the left-invariant vector field associated to the function $f\in C^\infty(S^1,\bbC)\cong \cX(S^1)$ is the vector field on $\cJ$ which assigns to a curve $\gamma(\theta)\in \cJ$ the element $f\gamma_\theta\in C^\infty(S^1,\bbC) = T_\gamma\cJ$, where as usual $\gamma_\theta$ denotes the derivative of $\gamma$ in the variable $\theta$.
The Lie bracket of vector fields on $\cJ$ is given by the well known formula
\[
\big[f,g\big](\gamma)
:=
\tfrac{d}{dt}\big|_{t=0}\,
g(\gamma+tf)
-
\tfrac{d}{dt}\big|_{t=0}\,
f(\gamma+tg).
\]
Specializing to left-invariant vector fields, this Lie bracket becomes:
\[
[f\gamma_\theta,
g\gamma_\theta]
=
\tfrac{d}{dt}\big|_{t=0} \Big(
g\cdot( \gamma+tf\gamma_\theta)_\theta
\Big)
-
\tfrac{d}{dt}\big|_{t=0} \Big(
f\cdot( \gamma+tg\gamma_\theta)_\theta
\Big)
= (gf_\theta-fg_\theta)\gamma_\theta .
\]
Left-invariant vector fields are closed under the Lie bracket, 
and their Lie bracket agrees with
the Lie bracket on $\cX(S^1)$ provided by the opposite of the usual Lie bracket of vector fields. This is exactly the Lie bracket on $\Witt$.


Every element of $\big(\bigwedge^k\cX(S^1)\big)^*$ gives rise to a left-invariant $k$-form on $\Ann^{\le \ast}$. We will show that, when restricted to left-invariant forms on $\Ann^{\le \ast}$, the de Rham differential agrees with the differential that computes the Lie algebra cohomology of the Witt algebra. Indeed, recall that the formula
\begin{align}\label{eqn:Cartan-Eilenberg}
(df)(x_0, \ldots, x_{k}) &= \sum_i (-1)^{i} x_i f(x_0, \ldots, \hat x_i, \ldots, x_{k})\\ &+ \sum_{i<j} (-1)^{i+j} f([x_i,x_j],x_0, \ldots, \hat x_i, \ldots, \hat x_j, \ldots, x_{k})\nonumber
\end{align}
for the differential which computes Lie algebra cohomology, is virtually identical with the formula
\begin{align}\label{eqn:de Rahm}
d\omega(V_0,\ldots,V_k) &=
\sum_i(-1)^{i} d_{V_i}(\omega(V_0,\ldots,\widehat{V_i},\ldots,V_k))\\
&+ \sum_{i<j}(-1)^{i+j}\omega([V_i,V_j],V_0\ldots,\widehat{V_i},\ldots,\widehat{V_j},\ldots,V_k)\nonumber
\end{align}
for the de Rham differential.
In our case, the Lie algebra cohomology is taken with trivial coefficients, which simplifies the right hand side of \eqref{eqn:Cartan-Eilenberg} to $\sum_{i<j} (-1)^{i+j} f(\allowbreak[x_i,x_j],\allowbreak x_0, \ldots, \hat x_i, \ldots, \hat x_j, \ldots, x_{k})$. Similarly, when the $V_i$ are left invariant vector fields, the right hand side of \eqref{eqn:de Rahm} simplifies to $\sum_{i<j}(-1)^{i+j}\omega([V_i,V_j],V_0\ldots,\widehat{V_i},\allowbreak\ldots, \widehat{V_j},\ldots,V_k)$. The two differentials \eqref{eqn:Cartan-Eilenberg} and \eqref{eqn:de Rahm} agree on the nose, so Lie algebra $k$-cocycles map to closed left invariant $k$-forms on $\Ann^{\le \ast}$.


Recall the Virasoro cocycle:
\begin{equation}\label{def: omega_Vir}
\omega_{Vir}\big(f(z)\tfrac{\partial}{\partial z},g(z)\tfrac{\partial}{\partial z}\big)
= \tfrac 1{12}\int_{S^1}\tfrac{\partial^3 f}{\partial z^3}(z)\,g(z)\,\tfrac{dz}{2\pi i}.
\end{equation}
Letting $\underline{\omega_{Vir}\!}\,$ be the closed left invariant $2$-form on $\Ann^{\le \ast}$
associated to the Virasoro cocycle, we define
\[
\tAnn:=\left\{(A,\gamma,a)\,\left|\,\begin{matrix}
A\in\Ann,\,a\in \bbC\,
\\\gamma:[0,1]\to \Ann^{\le A},\,\,\,\,\\ \gamma(0)=1,\,\gamma(1)=A\,\,\;\end{matrix}\right\}\right/\begin{matrix}
\textstyle\big(\gamma,a\big)\sim\big(\gamma',a + \int_k\underline{\omega_{Vir}}\big),\\[.5mm]
\text{$k$ {a homotopy from} $\gamma$ {to} $\gamma'$}
\end{matrix}
\] 
In future work, we will show that every unitary positive energy representations of the Virasoro algebra exponentiate to representations of $\tAnn$.
Representations of $\Vir$ with central charge $c$ (i.e. representations that map $C$ to the scalar $c\in\bbR_{\ge0}$) correspond to representations of $\tAnn$ where the central $\bbC\subset \tAnn$ acts by $z\mapsto e^{cz}$. 
Such representations descend to representations of the following semigroup:
\[
\tAnn_c:=\left\{(A,\gamma,a)\,\left|\,\begin{matrix}
A\in\Ann,\,a\in \bbC^\times\,                     
\\\gamma:[0,1]\to \Ann^{\le A},\,\,\,\,\\ \gamma(0)=1,\,\gamma(1)=A\,\,\;\end{matrix}\right\}\right/\begin{matrix}
\textstyle\big(\gamma,a\big)\sim\big(\gamma',a \cdot \exp(c {\int_k\underline{\omega_{Vir}\!}})\,\big),\\[.5mm]
\text{$k$ {a homotopy from} $\gamma$ {to} $\gamma'$}
\end{matrix}
\]
The semigroup $\tAnn_c$ may be in fact defined for any $c \in \bbC$.
When $c\neq0$ it is the quotient of $\tAnn$ by the central subgroup $(2\pi i/c)\bbZ\times\{0\}\subset \bbC\times\bbZ$, and when $c=0$ it is the product $\bbC^\times$ with the universal cover of $\Ann$.

\begin{prop}\label{prop: central extension}
The above semigroups fit into central extensions
\begin{equation}\label{eq: Central extension by C x Z}
0\to \bbC \times \bbZ \to \tAnn \to \Ann \to 0
\end{equation}
and
\begin{equation}\label{eq: Central extension by C^x x Z}
0\to \bbC^\times \times \bbZ \to \tAnn_c \to \Ann \to 0
\end{equation}
and there is a map $\tAnn\to \tAnn_c$ which restricts to $(z,n)\mapsto (e^{cz},n)$ on the centers.
\end{prop}

\begin{proof}
Let $\Ann^{(\infty)}$ denote the universal cover of $\Ann$.
There are obvious maps $\tAnn\to \Ann^{(\infty)}$ and $\tAnn_c\to \Ann^{(\infty)}$, given by only remembering the homotopy class of the path $\gamma$.
These maps are surjective by Corollary~\ref{cor: exist framings}, as any framing $h:S^1\times[0,1]\to A$  induces a path $\gamma:[0,1]\to \Ann^{\le A}$ by the formula $\gamma(t)=h(S^1\times[0,t])\subset A$ (with the boundary parametrisations given by $h|_{S^1\times\{0\}}$ and $h|_{S^1\times\{t\}}$).

Given a lift $A^{(\infty)}\in \Ann^{(\infty)}$ of an annulus $A$, we first argue that the fiber $F$ of $\tAnn\to \Ann^{(\infty)}$ over $A^{(\infty)}$ is isomorphic to $\bbC$.
Fix a path $\gamma_0:[0,1]\to \Ann^{\le A}$ representing $A^{(\infty)}$.

Any element of $F$ can be represented by a pair $(\gamma_0,a)$.
By definition, we then have $(\gamma_0,a)\sim(\gamma_0,a+\int_k\underline\omega_{Vir}\!)$ for every homotopy $k$ from $\gamma_0$ to itself (a map $S^2\to \Ann^{\le A}$).
Since $\underline{\omega_{Vir}\!}$ is a closed $2$-form and $\Ann^{\le A}$ has trivial second homotopy group (by Lemma~\ref{Ann le A is homotopy equivalent to S^1}) the integral of $\underline\omega_{Vir}\!$ over $k$ is trivial. So the map $a\mapsto (\gamma_0,a)$ is a bijection $\bbC\to F$.
Letting $K$ be the kernel of the map $\tAnn\to \Ann$,
we therefore get a commutative diagram:
\[
\tikz[xscale=2, yscale=1.2]{
\node[scale=1] (b') at (1,1) {$\bbC$};
\node[scale=1] (c') at (2.2,1) {$K$};
\node[scale=1] (d') at (3.3,1) {$\bbZ$};
\node[scale=1] (b) at (1,0) {$\bbC$};
\node[scale=1] (c) at (2.2,0) {$\tAnn$};
\node[scale=1] (d) at (3.3,0) {$\Ann^{(\infty)}$};
\node[scale=1] (C) at (2.2,-1) {$\Ann$};
\node[scale=1] (D) at (3.3,-1) {$\Ann$};
\draw[double, double distance=1, shorten >=1, shorten <=2] (b')--(b);
\draw[->, shorten >=1, shorten <=0] (c')--(c);
\draw[->, shorten >=1, shorten <=1] (d')--(d);
\draw[->, shorten >=2, shorten <=2] (b)--(c);
\draw[->, shorten >=2, shorten <=2] (c)--(d);
\draw[->, shorten >=2, shorten <=2] (b')--(c');
\draw[->, shorten >=2, shorten <=2] (c')--(d');
\draw[double, double distance=1, shorten >=2, shorten <=2] (C)--(D);
\draw[->] (c)--(C);
\draw[->] (d)--(D);
}
\]
where the middle horizontal row and the rightmost vertical column are central extensions.

We next argue that the middle vertical column is a central extension.
We wish to show that for any element $k\in K$ and $A\in \tAnn$, we have $kA=Ak$.
Write $kA = z_{A,k} Ak$, for some unique $z_{A,k} \in K$.
By connectedness of $\tAnn$, this element $z_{A,k}$ maps to zero in $\bbZ$ and therefore lives in $\bbC \subset K$.
In particular, $z_{A,k}$ is central in $\tAnn$.
Our goal is to show that $z_{A,k}$ is trivial.
We claim that $A\mapsto z_{A,k}$ is a homomorphism $\tAnn\to \bbC$, and that it descends to a homomorphism $\Ann^{(\infty)}\to \bbC$.
To see that $A\mapsto z_{A,k}$ is a homomorphism, we compute:
\[
kA_1A_2 = (kA_1)k^{-1}(kA_2) = (z_{A_1,k} A_1k)k^{-1}(z_{A_2,k} A_2k) =  z_{A_1,k}z_{A_2,k} A_1 A_2 k
\]
from which it follows that
$z_{A_1,k}z_{A_2,k}=z_{A_1A_2,k}$.\footnote{Throughout this proof, the group law on $(\bbC,+)$ is written multiplicatively.
What is really meant by this last equation is $z_{A_1,k}+z_{A_2,k}=z_{A_1A_2,k}$.
}
This descends to a homomorphism $\Ann^{(\infty)}\to \bbC$ because the kernel of $\tAnn\to \Ann^{(\infty)}$ is central.

We finish the argument by arguing that $\Ann^{(\infty)}$ admits no non-trivial homomorphisms to $\bbC$. Indeed, the subgroup $\Diff^{(\infty)}(S^1)$ has trivial abelianisation \cite[Thm.~1]{Thurston74}, and so any homomorphism $\Ann^{(\infty)} \to \bbC$ vanishes on this subgroup. 
Similarly, the finite dimensional subsemigroup $\text{M\"ob}^{(\infty)} \subset \Ann^{(\infty)}$ of universal-cover-lifted M\"obius annuli (i.e. annuli with boundary parametrisations in $\mathit{PSL}(2,\bbC)$) admits no non-trivial homomorphisms to any abelian group, and so any homomorphism $\Ann^{(\infty)} \to \bbC$ vanishes on this subsemigroup as well.
As $\Diff^{(\infty)}$ and $\text{M\"ob}^{(\infty)}$ together generate the dense subsemigroup of (universal-cover-lifted) thick annuli, we conclude that there are no nontrivial homomorphisms $\Ann^{(\infty)} \to \bbC$. It follows that the sequence $K\to \tAnn\to \Ann$ is a central extension, and that $K\cong \bbC\times \bbZ$.

This finishes the construction of the central extension \eqref{eq: Central extension by C x Z}.
For $c\neq0$, the central extension \eqref{eq: Central extension by C^x x Z} follows from \eqref{eq: Central extension by C x Z} by taking everywhere the quotient by the central subgroup $(2\pi i/c)\bbZ\subset \bbC$. Finally, if $c=0$, the central extension \eqref{eq: Central extension by C^x x Z} is visibly trivial.
\end{proof}

The semigroups $\tAnn$ and $\tAnn_c$ admit natural real diffeological structures, and the proof of Proposition~\ref{prop: central extension} works verbatim in the real diffeological category, showing that 
$\ker(\tAnn\to\Ann)=\bbC\times \bbZ$, and $\ker(\tAnn_c\to\Ann)=\bbC^\times\times \bbZ$.

Our next task is to show that the central extensions \eqref{eq: Central extension by C x Z} and \eqref{eq: Central extension by C^x x Z} are in fact central extensions of both real and complex diffeological semigroups (Definition~\ref{def: central extension of diffeological semigroups}).
In the real case, this is the content of the following proposition:

\begin{prop}\label{prop: it's locally trivial}
The central extensions \eqref{eq: Central extension by C x Z} and \eqref{eq: Central extension by C^x x Z}
are locally trivial in the sense that for every annulus $A\in\Ann$ there exists an open neighbourhood $U\subset \Ann$ of $A$
such that the pullback to $U$ is a trivial principal bundle.
\end{prop}

\begin{proof}
The map $\Ann^{(\infty)}\to\Ann$ is a covering space, hence locally trivial.
So it is therefore enough to argue that $\tAnn\to \Ann^{(\infty)}$ is locally trivial. We will in fact show that it is globally trivial.

Since $\Ann$ is homotopy equivalent to a circle, its universal cover $\Ann^{(\infty)}$ is contractible.
As in the proof of Corollary~\ref{cor: semi annuli contractible},
the homotopy equivalence between $\Ann$ and $S^1$ may be chosen smooth.
Consequently, the contracting homotopy of $\Ann^{(\infty)}$ may also be chosen smooth.

Let $\tilde A\mapsto \gamma_{\tilde A}:\Ann^{(\infty)} \to \Ann^{[0,1]}$ be a smooth contracting homotopy.
It assigns to an annulus $A\in \Ann$ and a lift $\tilde A\in \Ann^{(\infty)}$ a smooth path $\gamma_{\tilde A}:[0,1]\to \Ann$ that represents $\tilde A$. The section
\[
s:\tilde A \mapsto (A,\gamma_{\tilde A},0) : \Ann^{(\infty)} \to \tAnn
\]
provides the desired trivialisation.

The map $\bbC\times \Ann^{(\infty)} \to \tAnn$ is given by $(z,\tilde A)\mapsto z\cdot s(\tilde A)=(A,\gamma_{\tilde A},z)$. And the inverse map $\tAnn\to \bbC\times \Ann^{(\infty)}$ sends an element $[(A,\gamma,z)]\in \tAnn$ to the pair $(z+\int_k \underline{\omega_{Vir}},\tilde A)\in \bbC\times \Ann^{(\infty)}$, where $\tilde A=[(A,\gamma)]\in \Ann^{(\infty)}$, and $k$ is a homotopy from
$\gamma$ to $\gamma_{\tilde A}$.
\end{proof}

\begin{cor}\label{COR 1}
The central extensions \eqref{eq: Central extension by C x Z} and \eqref{eq: Central extension by C^x x Z} are central extensions of real diffeological semigroups (based on manifolds with corners).
\end{cor}

We conjecture that $\tAnn$ is a universal central extension of $\Ann$ in the category of real diffeological semigroups:

\begin{conj}
The central extension 
\[
0\to \bbC \times \bbZ \to \tAnn \to \Ann \to 0
\]
which appears in \eqref{eq: Central extension by C x Z} is a universal central extension in the category of diffeological semigroups over manifolds with corners.
\end{conj}

The main difficulty that we encountered while trying to prove this conjecture had to do with assigning a Fr\'echet Lie algebra to an arbitrary central extension
$Z \to G \to \Ann$. We suspect that the formalism of diffeological semigroups over manifolds with corners might not be the most suitable one for addressing this kind of question.

Our next goal is to show that both $\tAnn$ and $\tAnn_c$ have natural complex diffeologies, and that
\eqref{eq: Central extension by C x Z} and \eqref{eq: Central extension by C^x x Z} are central extensions of complex diffeological semigroups.
We begin by describing the complex diffeology on $\tAnn$.

\begin{defn}\label{def: holomorphic family of framings}
Let $M$ be a finite dimensional complex manifold, and let $\pi:\underline A \to M$ be a holomorphic family of annuli. A holomorphic family of framings of $\underline A$ is a smooth map $h : M \times S^1 \times [0,1] \to \underline A$ such that for every $m \in M$ the restriction of $h$ to $\{m\} \times S^1 \times [0,1]$ is a framing of $\pi^{-1}(m)$, and for every point $(\theta,t)\in S^1 \times [0,1]$ the restriction of $h$ to $M \times \{(\theta,t)\}$ is holomorphic.
\end{defn}

Let $f:M\to \Ann$ be a holomorphic map, represented by a holomorphic family of annuli $\underline A \to M$.
If $h$ is a holomorphic family of framings of $\underline A$, and $z:M \to \bbC$ is a holomorphic function, then the pair $(h,z)$ induces a lift $\tilde f:M\to \tAnn$.

\begin{defn}
Let $M$ be a finite dimensional complex manifold.
We declare a map $f:M \to \tAnn$ to be holomorphic if there exists an open cover $\{M_i\}$ of $M$ such that each restriction $f|_{M_i}$ is induced by a pair $(h,z)$ as above.
\end{defn}

\begin{remark}
We suspect that every holomorphic map $f:M \to \tAnn$ is in fact represented by a globally defined pair $(h,z)$ as above.
\end{remark}

Our final task in this section is to argue that $\tAnn \to \Ann$ is a central extension of complex diffeological semigroups.

\begin{prop}\label{prop: holomorphic local lifts}
Let $M$ be a finite dimensional complex manifold.
Then every holomorphic map $f:M \to \Ann$ locally lifts to a holomorphic map $M \to \tAnn$.
\end{prop}

\begin{proof}
For $m \in M$, we write $A_m$ for the annulus $f(m) \in \Ann$.
Fix a point $m_0 \in M$ and let $A:=f(m_0)\in\Ann$.
We assume without loss of generality that $M\subset \bbC^N$, and that $m_0=0$.
In this proof, we will construct a holomorphic family of framings of the annuli $A_m$ for $m\in U$, where $U\subset \bbC$ is a small open ball around the origin.
Such a family represents a local holomorphic lift $U \to \tAnn_c$ of our given map $f$.

When $A$ is completely thin (has empty interior), then the family $f$ is constant by Corollary~\ref{cor: hol family of thin is constant}, and there is nothing to show. We therefore assume that the thin part $Z := \partial_{in}A \cap \partial_{out}A$ 
is not the whole of $A$.

Pick an embedding $\iota:A\hookrightarrow \{z\in \bbC : \mathrm{Re} (z)\le 0\}$
such that $Z$ maps into $i\bbR$.
Let $D_1$ and $D_2$ be the closures of the two connected components of $\bbC P^1 \setminus \iota(A)$, equipped with the boundary parametrisations inherited from those of $A$.
By uniformizing the family $D_1 \cup A_m \cup D_2$, we get a family of embedding $A_m \hookrightarrow  \bbC$
that extends $\iota$.
From now on, we identify all the annuli $A_m$ with their images in $\bbC$ via the above embeddings, and we identify the total space of the family with a subset of $M\times \bbC$.
We write $Z_m\subset A_m$ for the thin part of $A_m$, and $\underline Z\subset M\times \bbC$ for the union of all the $Z_m$.

Let $\psi\in\Diff(S^1)$ be such that $\psi\circ \varphi_{out}( \theta) \in \varphi_{in}( \theta ) + \bbR$
in some open subset $W\subset \bbC$ containing $Z$, where $\varphi_{in/out}$ are the boundary parametrisations of $A$.
(More precisely, we require that the above condition holds for all $\theta\in S^1$ such that $\varphi_{in}( \theta )$ or $\psi\circ \varphi_{out}( \theta )$ lie in $W$.)
If we are able to produce a holomorphic lift $\tilde g$ of $g:=\psi \cdot f$, then for any lift $\tilde \psi \in \tDiff(S^1)$ of $\psi$, the function $\tilde \psi^{-1} \cdot \tilde g$ provides the desired holomorphic lift of $f$.
We therefore suppress the diffeomorphism $\psi$ and assume without loss of generality that
\begin{equation}\label{eq: phi in minus phi out}
\varphi_{out}( \theta ) - \varphi_{in}( \theta )\in\bbR_{\ge 0}
\end{equation}
holds true in $W$.
By shrinking $W$, we may furthermore assume that the horizontal line segment connecting $\varphi_{in}( \theta )$ and $\varphi_{out}( \theta )$ is contained in $A$.

By Corollary~\ref{cor: exist framings}, there exists a framing $h=h_0: S^1 \times [0,1] \to A \subset  \bbC$.
More precisely, the proof of Corollary~\ref{cor: exist framings} (which is really the proof of Proposition~\ref{prop: exist semi-framings}) provides a framing $h$ whose derivative is invertible on the preimage of $A \setminus Z$, and satisfies 
\[
h(\theta,t) = t \varphi_{out}(\theta) + (1-t) \varphi_{in}(\theta)
\]
in $W$. We assume that our framing $h_0$ satisfies those two properties.

Consider the following formulas:
\begin{align*}
h^{(1)}_m(\theta,t) &:= h_0(\theta,t) + t (\varphi_{m,out}(\theta)-\varphi_{0,out}(\theta)) + (1-t) (\varphi_{m,in}(\theta)-v_{0,in}(\theta)),\\
h^{(2)}_m(\theta,t) &:= t \varphi^{out}_m(\theta) + (1-t) \varphi^{in}_m(\theta),\\
\intertext{and note that}
h^{(1)}_0(\theta,t) &= h^{(2)}_0(\theta,t) = h_0(\theta,t).
\end{align*}
Since $h_0$ has invertible derivative in the complement of $W$, and having invertible derivative is an open condition, 
$h^{(1)}_m$ satisfies the condition of being a framing in the complement of $\underline W := M \times W$, after possibly shrinking $U$.
Further below, we will show that $h^{(2)}$
satisfies the condition of being a framing in a neighbourhood of $Z$, after possibly shrinking $U$.
Using a bump function $\lambda$ on $S^1$ which is zero on $\underline Z$ and $1$ on the complement of $\underline W$, we will then assemble the above `partial framings' to a globally defined function
\begin{equation}\label{eq: THE FINAL FRAMING}
h_m(\theta,t) = \lambda(\theta) h^{(1)}_m(\theta,t) + (1-\lambda(\theta))h^{(2)}_m(\theta,t)
\end{equation}
which is holomorphic as a function of $m$.
Note that this formula recovers $h_0(\theta,t)$ at $m=0$.
We claim that \eqref{eq: THE FINAL FRAMING} is a framing in a sufficiently small neighbourhood $U\subset M$ of $0$.
To check this, note that this agrees with $h^{(1)}_m(\theta,t)$ in a neighbourhod of $Z$, so we only need to worry about whether it's a framing away from the thin part. Once again, this follows from the invertibility of the derivative being an open condition,
and $h_0(\theta,t)$ having invertible derivative.

It remains to show that $h^{(2)}_m$ satisfies the condition of being a framing in a neighbourhood of the thin part of $A_m$.
Recall that $\varphi_{out}( \theta ) - \varphi_{in}( \theta )$ is assumed to be real-valued in 
a neighbourhood of the thin part, and that 
$h^{(2)}_m(\theta,t) = t \varphi^{out}_m(\theta) + (1-t) \varphi^{in}_m(\theta)$.

We will show below that for any fixed $\epsilon > 0$, upon shrinking $U$, we can arrange that
for every $x$ in our given neighbourhood of the thin part (minus the thin part itself) we have $|\arg(\varphi^{out}_m(x) - \varphi^{in}_m(x))| < \epsilon$.
Assuming this inequality, we can check that $h^{(2)}$ satisfies the condition of being a framing in a neighbourhood of the thin part:
\begin{align*}
\partial h^{(2)}/\partial t & = \varphi^{out}_m(\theta) - \varphi^{in}_m(\theta) \in e^{\pm i \varepsilon}\cdot \bbR_{\ge 0}
\\\partial h^{(2)}/\partial \theta & \approx i.
\end{align*}
Hence
\[
-\frac{\partial h^{(2)}/\partial t}{\partial h^{(2)}/\partial \theta} \in e^{\pm i 2\varepsilon}\cdot i \bbR_{\ge 0}
\]

It remains to check that, for any fixed $\epsilon > 0$, we can shrink $U$ so as to achieve $|\arg(\varphi^{out}_m(\theta) - \varphi^{in}_m(\theta))| < \epsilon$.
We assume wlog that $U$ is a ball centered at $0 \in \bbC^N$.

Recall that by assumption $\varphi_0^{out} - \varphi_0^{in}$ is $\bbR_{\ge 0}$-valued whenever $\varphi_0^{in/out}$ take values in $W$.
Without loss of generality, we assume that $W$ is a finite union of rectangles of the form $(-\epsilon,\epsilon) \times (a,b)$.
By suitably shrinking $U$, we may find functions 
\[
f^{in}_m, f^{out}_m:(a,b) \to (-\epsilon,\epsilon)
\]
such that $f^{in}_m\le f^{out}_m$, and such that the graphs $\Gamma_{f^{in}_m}$ and $\Gamma_{f^{out}_m}$ are contained in $\partial_{in}A_m$ and $\partial_{out}A_m$, respectively.
The inequality $f_m^{in} \le f_m^{out}$ implies that for any points $z_1\in\Gamma_{f_m^{out}}$ and $z_2\in\Gamma_{f_m^{in}}$
we have $z_1 - z_2 \not\in \bbR_{<0}$.
Thus, whenever $\varphi_m^{in}(\theta)$ and $\varphi_m^{out}(\theta)$ both lie in $W$, we have 
\[
\varphi^{out}_m(\theta)-\varphi^{in}_m(\theta)\not\in\bbR_{<0}.
\]
The function $m \mapsto \arg(\varphi^{out}_m(\theta) - \varphi^{in}_m(\theta))$ takes its values in $(-\pi, \pi)$, so
\begin{align*}
&\Phi_\theta:\,U\,\to\, \bbC\\[-1mm]
&\Phi_\theta(m):=\sqrt{\varphi^{out}_m(\theta) - \varphi^{in}_m(\theta)}
\end{align*}
is a holomorphic function with values in the right half-plane $\{z \in \bbC \, : \, \mathrm{Re}(z) \ge 0\}$.
Recall that we are trying to show that for any fixed $\varepsilon > 0$, there exists a small enough ball $U$ around $0$ such that for all $m \in U$ we have $ |\arg(\varphi^{out}_{m}(\theta) - \varphi^{in}_{m}(\theta))| < \varepsilon$ whenever $\varphi^{in/out}_m(\theta)$ lie in $W \setminus Z$.

Let $x\in \bbR_+$ be any point on the positive real line.
Given $\varepsilon > 0$, let $r=r(\varepsilon) > 0$ be such that the hyperbolic disc in the right half-plane centered around $x$ of hyperbolic radius $r$ is contained in the cone $\{z \in \bbC : |\arg(z)| < \varepsilon/2\}$.
Note that $r$ does not depend on $x$ by the scale invariance of the hyperbolic metric.
We have $\Phi_\theta(0)\in\bbR_{\ge0}$ by \eqref{eq: phi in minus phi out}. So, by the Schwarz-Pick lemma, the restriction of $\Phi_\theta$ to $rU\subset U$ takes its values in $\{z \in \bbC : |\arg(z)| < \varepsilon/2\}$, and its square
$\varphi^{out}_{m}(\theta) - \varphi^{in}_{m}(\theta)$ takes its values in $\{z \in \bbC : |\arg(z)| < \varepsilon\}$, as required.
\end{proof}

We believe that the proof of Proposition~\ref{prop: holomorphic local lifts} could be adapted to the case where $M$ is an infinite dimensional complex manifold.
Sadly, this falls short of proving that $\tAnn\to \Ann$ admits local holomorphic sections, because $\Ann$ is not a manifold.
Conveniently, our formulation of Definition~\ref{def: central extension of diffeological semigroups} 
does not require us to prove the existence of local holomorphic sections in order to get the following corollary:

\begin{cor}\label{COR 2}
The central extensions \eqref{eq: Central extension by C x Z} and \eqref{eq: Central extension by C^x x Z} are central extensions of complex diffeological semigroups.
\end{cor}

Assembling Corollaries~\ref{COR 1} and~\ref{COR 2}, we have shown:

\begin{thm}
The central extensions in Proposition~\ref{prop: central extension} can be naturally upgraded
to central extensions
\begin{equation}\label{eq: Central extension by C x Z diffeological}
0\to \bbC \times \bbZ \to \tAnn \to \Ann \to 0
\end{equation}
and
\begin{equation}\label{eq: Central extension by C^x x Z diffeological}
0\to \bbC^\times \times \bbZ \to \tAnn_c \to \Ann \to 0
\end{equation}
of real diffeological semigroups (based on the category of manifolds 
with corners), and also to central extensions of complex diffeological semigroups.
\end{thm}

\appendix

\section{Functions and forms on annuli}

Recall from \eqref{eq: def of O_A} the definition of the structure sheaf of an annulus $A \subset \bbC$.
There exists an alternative description of that sheaf which is better suited for doing differential geometry:

\begin{thm}\label{thm: alternative description of O_A}
Let $A \subset \bbC$ be an embedded annulus.
Let $U$ be an open subset of $A$ of the form $U=V \cap A$ for some open subset $V$ of $\bbC$. Then a function $f:U\to \bbC$ is in $\cO_A(U)$ if and only if $f$ is holomorphic in $\mathring U$, and there exists a smooth function $\hat f:V\to \bbC$ whose restriction to $U$ is $f$. 
\end{thm}

This alternative description will allow us to construct the tangent and cotangent bundles of an annulus in a way that makes it obvious that they are bundles.
Before proving Theorem~\ref{thm: alternative description of O_A}, we first prove a technical lemma:

\begin{lem}\label{lem: two defs of O_H}
Let $U\subset \HH$ be an open subset of the upper half-plane, and let $f:U\to \CC$ be a continuous function such that $f|_{U\cap \mathring\HH}$ is holomorphic and $f|_{U\cap \partial\HH}$ is smooth.
Then $f$ is smooth all the way to the boundary $U \cap \partial\HH$.
\end{lem}
\begin{proof}
Let $x\in U\cap \partial\HH$ be a point, and let $D\subset U$ be a neighbourhood of $x$ which is compact, simply connected, and with smooth boundary:
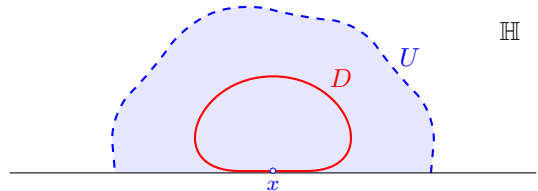
\[
\tikz[scale=.7]{
\draw[fill=blue!10, dashed, rounded corners=11, rotate=180, draw=blue, thick, scale=6] (0,0) +(180:.5) -- +(200:.55) -- +(220:.5) -- +(240:.55) -- +(260:.5) -- +(280:.55) -- +(300:.55) -- +(320:.5) -- +(340:.55) -- +(360:.5);
\draw (-5,0) -- (5,0);
\draw[red, thick, yshift=1, scale=.9] (-.7,0) -- (.7,0) to[out=0, in=0, looseness=2] (0,2) to[out=180, in=180, looseness=2] (-.7,0);
\node[red, scale=.9] at (1.4*.92,1.95*.92) {$D$};
\node[blue, scale=.9] at (2.6,2.2) {$U$};
\node[scale=.9] at (4.5,2.7) {$\HH$};
\draw[yshift=1.2, blue, ultra thin, fill=white] (0,0) circle (.05) node[below] {$\scriptstyle x$};
}
\]
We also assume that $[a,b]:=\partial D\cap \partial\HH$ is a connected interval.
Let $\psi:\DD\to D$ be a uniformizing map.
The function $g:=\psi^*f$ is continuous, holomorphic in the interior of $\DD$, and smooth on the boundary,
except possibly at the preimages of the two boundary points $a$ and $b$ of $\partial D\cap \partial\HH$.

Assuming that smoothness on the boundary on the boundary had been shown everywhere,
the Taylor coefficients $a_n$ of $g(z)=\sum a_nz^n$ would satisfy
\[
\begin{split}
a_n&\textstyle =\,\tfrac1{2\pi i}\oint_{|z|=r} g(z)z^{-(n+1)}dz\qquad\quad\text{for any $r<1$}\\
&\textstyle =\,\tfrac1{2\pi i}\oint_{|z|=1} g(z)z^{-(n+1)}dz\qquad\quad\text{since $g$ is continuous}\\
&\textstyle =\,\tfrac{\pm1}{n(n-1)\ldots(n-k+1)}\cdot\tfrac1{2\pi i}\oint_{|z|=1} \big(g|_{\partial\DD}\big)^{(k)}(z){\cdot} z^{-(n-k+1)}dz\qquad \forall k\le n.
\end{split}
\]
It would follow that $|a_n|\le \tfrac1{n(n-1)\ldots(n-k+1)}\cdot\|(g|_{\partial\DD})^{(k)}\|_\infty$.
The coefficients $a_n$ would decay faster that any power of $n$, so $g(z)$ would be smooth all the way to the boundary.
The same would therefore hold for $f$ around $x$.

It remains to show that $f|_{\partial D}$ is smooth at the two boundary points of the interval $[a,b]:=\partial D\cap \partial\HH$.
Let $h\in\cO_\HH(\HH)$ be an auxiliary function with zeros of infinite order at $a$ and $b$ (for example, $h(z)=e^{-\frac{1+i}{\sqrt {z-a}}-\frac{1+i}{\sqrt {z-b}}}$).
We run the same argument as above with $\tilde f:=hf$ (the function $\tilde f|_{\partial D}$ is now smooth at $a$ and $b$ because it vanishes to infinite order), deduce that $\tilde f$ is smooth all the way to the boundary, and divide by $h$ to get the result.
\end{proof}

We next verify Theorem~\ref{thm: alternative description of O_A} in the case where $f$ is globally defined.

\begin{lem}\label{lem: alternative description of O_A(A)}
Let $A \subset \bbC$ be an embedded annulus, and let $f:A\to \bbC$ be a continuous function which is holomorphic in the interior, and smooth when restricted to each one of the two boundaries of $A$.
The there exists a smooth function $\hat f:\bbC\to \bbC$ whose restriction to $A$ is $f$. 
\end{lem}
\begin{proof}
For $c\in\bbR_{>0}$ large enough, the restriction of $f_1(z):=f(z)+cz$ to $\partial_{in} A$ is an embedding.
Let $D_{out}\subset \CC\cup\{\infty\}$ be the closed `outer' disc bound by $\partial_{in}A$ (the one containing $\infty$),
and let $D_{in}\subset \CC$ be the closed disc bound by $f_1(\partial_{in} A)$.
Let $\phi:=f_1|_{\partial D_{out}}:\partial D_{out}\to \partial D_{in}$, and let $\bbC P^1_\phi:=D_{out}\cup_\phi D_{in}$.
Note that $A\subset D_{out}$. 
Therefore $A\cup_\phi D_{in}\subset \bbC P^1_\phi$.
\noindent
Pick an isomorphism $\psi:\bbC P^1_\phi\to \bbC P^1$.
By the Riemann mapping theorem for simply connected domains with smooth boundary,
the function $\psi|_{D_{out}}$ is holomorphic in the interior and smooth all the way to the boundary.
Let
\[
\hat D:=\psi(A\cup_\phi D_{in})\subset \CC.
\]
The function $f_1:A\to\CC$ and the inclusion map $\iota:D_{in}\hookrightarrow \CC$ assemble to a function $f_1\cup \iota:A\cup_\phi D_{in}\to \CC$.
Let
\[
f_2:=(f_1\cup \iota)\circ\psi^{-1}:\hat D\to \CC.
\]
By Lemma~\ref{lem: two defs of O_H}, the function $f_2$ is holomorphic in the interior and smooth all the way to the boundary.
Pick a $C^\infty$ extension $\hat f_2:\bbC P^1\to \CC$ of $f_2$ and let $f_3:=(\hat f_2\circ\psi)|_{D_{out}}$.
Note that $f_3$ agrees with $f_1$ on $A$, and that it is smooth all the way to the boundary.
Finally, pick a $C^\infty$ extension $\hat f_3:\CC\to\CC$ of $f_3$, and let $g(z):=\hat f_3(z)-cz$.
\end{proof}

\begin{proof}[Proof of Theorem~\ref{thm: alternative description of O_A}]
Let $f:U\to \bbC$ be as in the statement of the theorem.
Pick an open cover $\{U_i\}$ of $U$, and pick embedded annuli $U_i\subset A_i\subset \bbC$ such that $\mathrm{closure}(\mathring A_i)\subset U$.\footnote{Note that we do not require $A_i\subset A$, only $\mathring A_i\subset A$.}
Let $f_i:A_i\to \bbC$ be any function extending $f|_{\mathrm{closure}(\mathring A_i)}$ and which is smooth on the thin part of $A_i$.
By Lemma~\ref{lem: alternative description of O_A(A)}, $f_i$ admits a smooth extension $\hat f_i:\bbC\to\bbC$ to the whole of $\bbC$.
Let $\phi_i:V\to \bbR_{\ge 0}$ be a partition of unity with the property that $\mathrm{support}(\phi_i|_U)\subset U_i$.
Then $\sum \phi_i\hat f_i$ is the desired smooth extension of $f$.
\end{proof}

With the above alternative definition of the structure sheaf of an embedded annulus, we can effortlessly talk about derivatives:

Let $f:A\to\bbC$ be a holomorphic function (by which we mean a section of the sheaf~\eqref{eq: def of O_A}).
We then have the following three functions, defined on various parts of $A$:
\begin{equation}\label{eq: three defs of the derivative}
\tfrac{\partial}{\partial z}\big(f|_{\mathring A}\big):\mathring A\to \bbC,
\qquad
\tfrac{\partial}{\partial z}\big(f|_{\partial_{in}A}\big):\partial_{in}A\to \bbC,
\qquad
\tfrac{\partial}{\partial z}\big(f|_{\partial_{out}A}\big):\partial_{out}A\to \bbC.
\end{equation}
The first one is defined in the usual way.
The last two are defined using local parametrisations of $\partial_{in}A$ and of $\partial_{out}A$, as above.

\begin{lem}\label{lem: derivative}
Let $A\subset \bbC$ be an embedded annulus, $U\subset A$ open, and let $f:A\to \bbC$ be a holomorphic function (an element of $\cO_A(U)$).
Then the three functions \eqref{eq: three defs of the derivative} assemble to a holomorphic function $f'\in\cO_A(U)$.
\end{lem}
\begin{proof}
Let $\hat f:V\to \bbC$ be a smooth function defined on some open subset $V$ of $\bbC$, whose restriction to $U$ is $f$.
Then $f'$ is the restriction to $U$ of the function $\tfrac{\partial}{\partial z}\hat f:=\tfrac12(\frac{\partial}{\partial x}-i\frac{\partial}{\partial y})(\hat f)$.
\end{proof}

If $A$ is an abstract annulus, and $f:U\to \bbC$ is a function defined on some open subset of $A$, then
its derivative $f'$ is not a well-defined function.
But given two functions $f,g:U\to \bbC$, the ratio $f'/g'$ is well defined (provided $g'$ is nowhere zero) in the sense that it doesn't depend on the embedding $A\hookrightarrow \bbC$.

\begin{defn}\label{def: tangent bundle of annulus}
The cotangent bundle $T^*A$ of an annulus is the holomorphic line bundle 
whose local holomorphic sections on some open $U\subset A$ are formal expressions of the form $fdg$, for $f,g\in \cO_A(U)$ and $g'$ nowhere zero, modulo the equivalence relation $f dg_1 \sim f(g_1'/g_2') dg_2$.

The tangent bundle $TA$ is similarly defined as the line bundle whose local holomorphic sections are expressions of the form $f\partial_g$ for $f,g\in \cO_A(U)$ and $g'$ nowhere zero, modulo the equivalence relation $f \partial_{g_1} \sim f(g_2'/g_1') \partial_{g_2}$.
\end{defn}

With this notion of tangent bundle of an annulus, and recalling the definition of tangent space of $\Ann$ given in Definition~\ref{defn: tangent spaces with corners},
we can compute the tangent space of the semigroup of annuli at a given annulus $A$ \cite[Prop 2.4]{SegalDef}:

\begin{prop}\label{prop: tangent space of Ann}
The tangent spaces of the semigroup of annuli are given by
\[
T_A\Ann
\,=\,\frac{\cX(\partial_{out}A)\oplus \cX(\partial_{in}A)}{\cX_{\mathrm{hol}}(A)}
\,=\,\frac{\cX(S^1)\oplus \cX(S^1)}{\cX_{\mathrm{hol}}(A)}.
\]
Here, $\cX(S^1)$ denotes the space of complexified vector fields on the circle,
and $\cX_{\mathrm{hol}}(A)$ denotes the space of holomorphic sections of $TA$.
\end{prop}

\begin{proof}
Let $M:=\{\psi_\pm:\DD_\pm\hookrightarrow \bbC P^1\,|\,\psi_\pm\text{ holomorphic }, \psi_-(z)=z+\mathcal O(z^{-1})\}$,
let $N:=N=\big\{\gamma_\pm:S^1\to \bbC\,|\,\text{$\gamma_\pm$ has winding number $\pm1$}\}$, and
let $\mathrm{Emb}(A)$ denote the space of holomorphic embeddings of $A$ into the complex plane (such that $\partial_{in}(A)$ is the boundary of the bounded component of $\bbC\setminus A$).
All three spaces are manifolds (as opposed to manifolds with boundary/corners), 
and there are two obvious maps $s:M\to N$ and $s':\mathrm{Emb}(A)\to N$ given by restriction to the boundary.

There is also a retraction $r:N\to \mathrm{Emb}(A)$ of $s'$ constructed as follows.
A pair $(\gamma_+,\gamma_-) \in N$ defines a pair of discs with parametrised boundary $D_{\pm}\subset \bbC P^1$, which we may use to form the genus zero surface $D_-\cup A\cup D_+$.
The uniformizing map $D_-\cup A\cup D_+ \to \bbC P^1$ that sends $\infty\in D_-$ to $\infty\in \bbC$ and respects the second order jets restricts to an embedding $\sigma:A\hookrightarrow\bbC P^1$.
We define $r(\gamma_+,\gamma_-)$ to be that embedding.

Note that the composite $r\circ s$ is constant, and that the fiber of $r$ over that point is exactly $M$.

Using the model of $\Ann$ given by
\[
\Ann=\big\{ \psi_\pm \in M \big |\, \psi_-(\mathring\DD_-)\cap \psi_+(\mathring\DD_+)=\emptyset \big\},
\]
we may identify the tangent space $T_A\Ann$ with the usual tangent of the manifold $M$.\\

\indent
At the level of tangent spaces, the diagram
\[
\tikz{
\node (a) at (0,0)
{$M$};
\node (b) at (2,0)
{$N$};
\node (c) at (4,0)
{$\mathrm{Emb(A)}$};
\draw[->] (a) --node[above, yshift=-2]{$\scriptstyle s$} (b);
\draw[->] (b) --node[above, yshift=-2]{$\scriptstyle r$} (c);
\draw[shorten <=-2, shorten >=-2, ->] (c) to[bend right=35]node[above, yshift=-2]{$\scriptstyle s'$} (b);
}
\]
induces a split short exact sequence
\[
0\to   T_A\Ann
\to
\Gamma\big(S^1\sqcup S^1,(\gamma_-\sqcup\gamma_+)^*T\bbC P^1\big)
\to
\Gamma_{\mathrm{hol}}\big(A,\sigma^*T\bbC P^1\big)   \to   0.
\]
The result follows since
$\Gamma(S^1\sqcup S^1,(\gamma_-\sqcup\gamma_+)^*T\bbC P^1)=
\Gamma(S^1,T_\bbC S^1)^{\oplus 2}=\cX(S^1)\oplus \cX(S^1)$,
and
$\Gamma_{\mathrm{hol}}(A,\sigma^*T\bbC P^1)=
\Gamma_{\mathrm{hol}}(A,TA)=\cX_{\mathrm{hol}}(A)$.
\end{proof}

For future purposes, we record a version of Cauchy's theorem adapted to our situation:

\begin{prop}[Cauchy's theorem]\label{prop: Cauchy's theorem}
Let $A$ be an annulus, and let $\omega$ be a holomorphic 1-form on $A$, i.e. a holomorphic section of $T^*A$.
Then
\[
\int_{\partial_{in}A}\omega\,\,=\,\,\int_{\partial_{out}A}\omega.
\]
\end{prop} 
\begin{proof}
Let $A_1 \subset A_2 \subset \cdots  A_n \subset \cdots A$ be a sequence of annuli such that
$\partial_{in/out} A_n$ converges to $\partial_{in/out} A$ as smooth curves,
and such that
\[
C_n:=\partial_{in} A_n\setminus(\partial_{in} A_n\cap \partial_{out} A_n)
\qquad\text{and}\qquad
C_n':=\partial_{out} A_n\setminus(\partial_{in} A_n\cap \partial_{out} A_n)
\]
have their closures in $\mathring A$.
Intuitively, $A_n$ is obtained from $A$ by pushing the boundary of the thick part of $A$ into the interior of $A$.
Then 
\[
\begin{split}
\int_{\partial_{in}A}\!\omega=
&\lim_{n\to \infty} \int_{\partial_{in}A_n}\!\omega
=
\lim_{n\to \infty} \left[\int_{C_n}\!\omega
\,+\int_{\partial_{in} A_n\cap \partial_{out} A_n}\!\omega
\right]\\
&=\lim_{n\to \infty} \left[\int_{C_n'
}\!\omega
\,+\int_{\partial_{in} A_n\cap \partial_{out} A_n}\!\omega
\right]=
\lim_{n\to \infty} \int_{\partial_{out}A_n}\!\omega=
\int_{\partial_{out}A}\!\omega,
\end{split}
\]
where the first and last equal signs hold because $\omega$ is continuous, and
$\int_{C_n}\!\omega$ equals $\int_{C_n'}\!\omega$ by the usual Cauchy's theorem.
\end{proof}

We end this section with a description of holomorphic vector field on annuli from the perspective of framings.
Given an annulus $A$ equipped with a framing $h:S^1 \times [0,1] \to A$,
and given a holomorphic vector field $v$ on $A$ (a section of the tangent bundle $TA$), we may pull it back to a complexified vector field on $S^1 \times [0,1]$ of the form $f(\theta,t) \partial_\theta$ via the formula
\[
h^*(v)_{(\theta,t)}  :=  (d(h|_{S^1})_\bbC)^{-1} (v_{h(\theta,t)})  \partial_\theta
\]
where $d(h|_{S^1})_\bbC : \bbC \to T_{h(\theta,t)}A$ is the $\bbC$-linear extension of $d(h|_{S^1}) : \bbR = T_\theta S^1 \to T_{h(\theta,t)}A$

With an eye to future applications, we wish to characterise those vector fields $f(\theta,t) \partial_\theta$ on $S^1 \times [0,1]$ that occur as the pullback of a holomorphic vector field on $A$.

\begin{lem}\label{lem: X-holomorphic vector field}
Let $h:S^1 \times [0,1] \to A$ be a framed annulus, and let $X = -\tfrac{h_t}{h_\theta}$.
If $f(\theta,t) \partial_\theta$ is of the form $h^*(v)$ for some holomorphic vector field on $A$,
then 
\[
f_t = X_\theta f - X f_\theta.
\]
\end{lem}

\begin{proof}
We may assume that the annulus $A$ is embedded in the complex plane. Then $f$ is given in terms of $v$ by the formula
\[
f(\theta,t) = h_\theta(\theta,t)^{-1} v(h(\theta,t))
\]
Differentiating in $t$ gives:
\begin{align*}
 f_t &= ( \partial_t[v\circ h] \cdot h_\theta - h_{\theta t} \cdot v\circ h) ) / h_\theta^2\\
   &= ( v_z\circ h \cdot h_t h_\theta +  \partial_\theta [X h_\theta] \cdot v\circ h ) / h_\theta^2\\
   &= ( - v_z\circ h \cdot X \cdot h_\theta^2 +  (X_\theta h_\theta + Xh_{\theta\theta})\cdot v\circ h ) / h_\theta^2
\end{align*}
Differentiating in $\theta$ gives:
\begin{align*}
f_\theta(\theta, t) 
&= ( \partial_\theta[v\circ h] \cdot h_\theta - h_{\theta \theta} \cdot v\circ h) / h_\theta^2\\
&= (v_z \circ h \cdot h_\theta^2 - h_{\theta\theta} \cdot v \circ h)/h_\theta^2
\end{align*}
We now check:
\begin{align*}
X_\theta f - X f_\theta
&= X_\theta \cdot v\circ h/ h_\theta - X \cdot ( v_z\circ h \cdot h_\theta^2 - h_{\theta\theta} \cdot v\circ h ) / h_\theta^2\\
&= (X_\theta \cdot v\circ h \cdot h_\theta - X \cdot ( v_z\circ h \cdot h_\theta^2 - h_{\theta\theta} \cdot v\circ h ) ) / h_\theta^2 = f_t
\end{align*}
\end{proof}

\def\lfhook#1{\setbox0=\hbox{#1}{\ooalign{\hidewidth
  \lower1.5ex\hbox{'}\hidewidth\crcr\unhbox0}}}

\end{document}